\documentclass[reqno,10pt]{amsart}

\usepackage{amsmath}%
\usepackage{rotating}
\usepackage{epstopdf, pinlabel, multirow} 
\usepackage{amsfonts}%
\usepackage{amssymb}%
\usepackage{graphicx}
\usepackage{xy,amsthm,enumerate,xypic,array}  

\usepackage{wrapfig}

\newcommand{\ds}{\displaystyle}

\newcommand{\Z}{\mathbb{Z}}

\newcommand{\tensor}{\otimes}

\newcommand{\op}{\mathcal}
\newcommand{\fr}{\mathfrak}

\newcommand{\cdc}{,\dots,}
\newcommand{\tdt}{\tensor\dots\tensor}

\newcommand{\cupp}{\mathbin{\smile}}

\newcommand{\TT}{\text{\sffamily{T}}}
\newcommand{\TA}{\text{\sffamily{M}}}

\newcommand{\B}{\text{\sffamily{B}}}
\newcommand{\cB}{\text{\sffamily{B}}_\circlearrowright}
\newcommand{\spB}{\text{\sffamily{B}}_\dagger}

\newcommand{\ZA}{\text{\sffamily{M}}_{\circlearrowright}}
\newcommand{\ZZ}{\text{\sffamily{T}}_{\circlearrowright}}

\newcommand{\nsl}[1]{{\op{{#1}}^\ast}}
\newcommand{\syl}[1]{{\overline{\op{{#1}}}^\ast}}
\newcommand{\nsc}[1]{{\op{{#1}}_\ast}}
\newcommand{\syc}[1]{{\overline{\op{{#1}}}_\ast}}

\newcommand{\cnsl}[1]{{\op{{#1}}^\bullet}}
\newcommand{\csyl}[1]{{\overline{\op{{#1}}}^\bullet}}
\newcommand{\cnsc}[1]{{\op{{#1}}_\bullet}}
\newcommand{\csyc}[1]{{\overline{\op{{#1}}}_\bullet}}

\newcommand{\kk}{\Bbbk}

\input{xy}
\xyoption{all}

\usepackage{fullpage}

\numberwithin{equation}{section}

\newtheorem{theorem}{Theorem}[section]

\newtheorem{corollary}[theorem]{Corollary}
\newtheorem{lemma}[theorem]{Lemma}
\newtheorem{proposition}[theorem]{Proposition}

\newtheorem*{thma}{Theorem A}
\newtheorem*{cora}{Corollary}
\newtheorem*{thmb}{Theorem B}
\newtheorem*{thmc}{Theorem C}

\theoremstyle{definition}
\newtheorem{definition}[theorem]{Definition}
\newtheorem{example}[theorem]{Example}

\newtheorem{remark}[theorem]{Remark}

\newtheorem*{q1}{Question 1}
\newtheorem*{q2}{Question 2}

\allowdisplaybreaks[2]

\addtocounter{MaxMatrixCols}{2}

\begin{document}

\title{Maurer-Cartan Elements and Cyclic Operads}
\author
[Benjamin C. Ward]{Benjamin C. Ward}
\email{bward@scgp.stonybrook.edu}

\address{Simons Center for Geometry and Physics, Stony Brook, NY 11794}

\begin{abstract}
First we argue that many BV and homotopy BV structures, including both familiar and new examples, arise from a common underlying construction.  The input of this construction is a cyclic operad along with a Maurer-Cartan element in an associated Lie algebra.  Using this result we introduce and study the operad of cyclically invariant operations, with instances arising in cyclic cohomology and $S^1$ equivariant homology.  We compute the homology of the cyclically invariant operations; the result being the homology operad of $\op{M}_{0,n+1}$, the uncompactified moduli spaces of punctured Riemann spheres, which we call the gravity operad after Getzler.  Motivated by the line of inquiry of Deligne's conjecture we construct `cyclic brace operations' inducing the gravity relations up-to-homotopy on the cochain level.  Motivated by string topology, we show such a gravity-BV pair is related by a long exact sequence.  Examples and implications are discussed in course.
\end{abstract}



\maketitle

\section*{Introduction}

Gerstenhaber algebras and Batalin-Vilkovisky (BV) algebras are structures fundamental to deformation theory.  They were introduced respectively by their namesakes in \cite{G} in the setting of Hochschild cohomology operations and in \cite{BV1} in the study of gauge fixing in quantum field theory, and have proven deserving of extensive study.  Consider a construction which takes for input some algebraic or topological object and produces as output a Gerstenhaber algebra (resp. $\op{G}_\infty$-algebra).  One can and should ask the following question:

\begin{q1}  What additional structure on the input of said construction would endow the output with a compatible BV operator?
\end{q1}

The compatibility requirement is $(-1)^{|a|}\{a,b\}=\Delta(ab)-\Delta(a)b-(-1)^{|a|}a\Delta(b)$.  Here are some examples of {\bf answers} to Question 1.
\begin{itemize}
\item  Let $M$ be a smooth manifold and let $\bigwedge \Gamma (T_\bullet M)$ be its space of polyvector fields.  The Schouten bracket with the exterior product make this space a Gerstenhaber algebra.  If $M$ comes equipped with a {\bf volume form} $\omega$, this Gerstenhaber algebra has a compatible BV operator via contraction with $\omega$.
\item  More generally, the sections of the exterior bundle of a Lie algebroid $A$ form a Gerstenhaber algebra.  If $A$ is equipped with a {\bf flat connection} on its determinant bundle then the associated covariant derivative induces a compatible BV operator \cite{Xu}.  See also the work of Kowalzig and Kr{\"a}hmer \cite{KoKr}, \cite{Kow} for a substantial generalization. 
\item  Let $A$ be an associative algebra.  Then the Hochschild cohomology $HH^\ast(A,A)$ is a Gerstenhaber algebra.  If $A$ is unital and comes equipped with an invariant symmetric nondegernate {\bf inner product}, Connes' boundary operator induces a compatible BV operator \cite{Menichi},\cite{Tradler}.  More generally we may consider a (non-symmetric) Frobenius algebra with semi-simple Nakayama automorphism \cite{Lambre}.
\item  In \cite{BG} Baranovsky and Ginzburg show that for a smooth complex Poisson variety $X$ with smooth coisotropic subvarieties $Y,Z$, $\text{Tor}^{\op{O}_X}(\op{O}_Y,\op{O}_Z)$ is a Gerstenhaber algebra.  They show a {\bf first order deformation} of the structure sheaf $\op{O}_X$ to a sheaf of noncommutative algebras and of $\op{O}_Y$ and $\op{O}_Z$ to sheaves of modules over the deformed algebra gives rise to a compatible BV operator.  See also \cite{BeF}.
\item  Let $X$ be a topological space with singular cochains $S^\ast(X)$.  Then $S^\ast(X)$ is a $\op{G}_\infty$-algebra, with product the cup product and with bracket vanishing on cohomology.  A {\bf circle action} on $X$ induces a compatible BV operator.
\end{itemize}

The first objective of this paper is to answer Question 1 in a context which is sufficiently general as to include several of the examples above.  In order to describe the construction of a $\op{G}_\infty$ algebra which provides this context, let us reconsider \cite{G} with a half century of hindsight.  

One starts with an associative algebra $A$ and considers the multilinear maps on such an algebra along with the operation of insertion of functions
\begin{align*}
Hom(A^{\tensor n}, A)\tensor Hom(A^{\tensor m},A)&\stackrel{\circ_i}\longrightarrow Hom(A^{\tensor n+m-1},A) \\
f\tensor g & \mapsto \left\{\text{insert $g$ into the }  i^{th} \text{ input of }f\right\}
\end{align*}
Taking the sum of all possible insertions produces a nonassociative bilinear operation whose commutator is a Lie bracket.  To this Lie algebra one can associate the Maurer-Cartan equation whose solutions parametrize differentials on the space of all multilinear maps.  Given that the algebra is associative, there is a canonical solution, given by the multiplication $\mu_2\colon A^{\tensor 2}\to A$.  The associated cochain complex is the Hochschild complex of $A$, and its cohomology comes with a Lie bracket {\it ens causa sui}.  Finally, a suitable ternary version of the Lie bracket can be formed, call it $B$ for `brace operation', such that the bilinear operation $B(\mu_2,-,-)$ serves as a cup product, which is associative and commutative on the level of cohomology.  The Lie bracket and cup product combine to form a Gerstenhaber algebra on the cohomology, and this structure is a Gerstenhaber algebra up to homotopy (in a very precise sense) on the cochain level.

\subsection*{Summary of Results.}

The above rendering of the Gerstenhaber structure of $HH^{\ast}(A,A)$ allows for a far reaching generalization of Gerstenhaber's original constructions, as well as the chain level constructions prompted by Deligne's conjecture.  In particular we see the insertion operations as fundamental.  The notion of an operad provides an axiomatization of these operations, and we are led to the following theorem.  For an operad $\op{O}$ we construct an associated Lie algebra $\nsl{O}$ and denote the Maurer-Cartan set by $MC(\nsl{O})$ and denote the differential twisted by a MC element $\zeta$ by $\delta_\zeta$.

\begin{thma}  Let $\op{O}$ be an operad in the category of differential graded vector spaces, and let $\zeta\in MC(\nsl{O})$.  Then the complex $(\nsl{O}, \delta_\zeta)$ is an algebra over a chain model for the little disks operad.  In particular, in characteristic $0$, $(\nsl{O}, \delta_\zeta)$ is a $\op{G}_\infty$-algebra.
\end{thma}

We will give a proof of Theorem A, but we hasten to add that such a proof is a matter of course in the wake of the vanguard who faced related problems (i.e. Deligne's conjecture) in recent decades, see below for a discussion of this.  In particular, the Maurer-Cartan equation in this context is represented by the $\op{A}_\infty$ operad, and the minimal operad of Kontsevich and Soibelman \cite{KS} provides the chain model.  Indeed, the first true aim of this paper is to answer Question 1 in light of Theorem A.  The answer which we provide says that a compatible BV operator arises when the operad is a cyclic operad and when the MC element is cyclically invariant.  In particular we prove the following theorem.

\begin{thmb}  Let $\op{O}$ be a unital cyclic operad in the category of differential graded vector spaces, and let $\zeta\in MC(\nsl{O})$ which is cyclically invariant.  Then the complex $(\nsl{O}, \delta_\zeta)$ is an algebra over a chain model for the framed little disks operad.  In particular, in characteristic $0$, $(\nsl{O}, \delta_\zeta)$ is a $\op{BV}_\infty$-algebra.
\end{thmb}

Thus, when the operad in question happens to be cyclic, Theorem B describes the natural operations on a complex associated to the underlying operad.  Now, however, one can also consider the complex of cyclic (co)invariants, which we denote $(\cnsl{O},\delta_\zeta)$.  We are then prompted to ask:

\begin{q2}  What are the natural operations on the complex of cyclic (co)invarants associated to a cyclic operad?  And how is this algebra related to the BV algebra of Theorem B?
\end{q2}

To answer Question 2 we construct a new operad denoted $\cB$ and called the cyclic brace operad, the cyclic analog of the brace operations.  The cyclic brace operad contains the $\op{L}ie$ operad as a suboperad and so one can apply the Maurer-Cartan formalism to produce an operad of natural operations which acts on any cyclic operad after choice of MC element.  Calculating the homology of these natural operations we found the homology of the uncompactified moduli spaces of punctured Riemann spheres $\op{M}_{0,n+1}$.  After Getzler we call the operad $\Sigma H_\ast(\op{M}_{0,n+1})$ the gravity operad $\op{G}rav$.  

\begin{thmc}  Let $\op{O}$ be a cyclic operad in the category of differential graded vector spaces, and let $\zeta\in MC(\cnsl{O})$.  Then the complex of cyclic (co)invariants $(\op{O}^\bullet,\delta_\zeta)$ is an algebra over a dg operad computing the homology of $\op{M}_{0,\ast+1}$. 
In particular the cohomology of the complex of cyclic (co)invariants is a gravity algebra.  Moreover the failure of the gravity relations on the cochain level is measured by explicitly constructed homotopies.  The gravity and BV algebras associated to such data are related by a long exact sequence constructed from the inclusion $\cnsl{O}\to\nsl{O}$.  
\end{thmc}

Theorem C is the central result of this paper and it may be interpreted as an $S^1$-equivariant version of Deligne's conjecture (see Corollary $\ref{equiDel}$ and the ensuing discussion).  The chain operad which we construct is the cyclic analog of the minimal operad of \cite{KS} and is denoted $\ZA$.  In particular the constituent spaces $\ZA(n)$ serve as new combinatorial chain models for the moduli spaces $\op{M}_{0,n+1}$.  An example of a consequence of Theorem C is the following.

\begin{cora}  The cyclic cohomology of a Frobenius algebra or a cyclic $A_\infty$ algebra is a gravity algebra.  This structure is induced at the cochain level by the action of $\ZA$ on Connes' $C_\lambda^\ast$ complex.
\end{cora} 

So {\bf in summary,} starting from a cyclic operad and a Maurer-Cartan element we construct a BV algebra and a gravity algebra as the cohomology of a pair of complexes, with an explicit up-to-homotopy structure, that fit together in a long exact sequence.  Examples of this construction arise from {\it eg} Frobenius algebras, symplectic dialgebras, string topology, deformation complexes, Fukaya categories and $S^1$-spaces.  Examples are discussed in Section $\ref{exsec}$. 

\subsection*{Connections with the literature.}

Theorem A is a generalization of several results in the literature.  Most notably when $\op{O}$ is the endomorphism operad of an associative algebra this result recovers the original Deligne conjecture.  For a summary of the history of this conjecture and its proofs one may consult the MathSciNet review of \cite{MS} written by A.A. Voronov.  If $\op{O}$ is the endomorphism operad of an $A_\infty$ algebra this result recovers a generalization proved in \cite{KS} and \cite{KSch}.  If the MC element is homogenous with respect to arity, this result recovers the generalization of \cite{MS} for operads with multiplication.  Additional examples of this theorem have been proven when $\op{O}$ is of the form $\op{P}^!\tensor End_A$, for a non-$\Sigma$ Koszul operad $\op{P}^!$ and a $\op{P}$-algebra $A$, in \cite{MM}, \cite{YauD}, and \cite{ValD} which give results in the language of operadic cohomology.  The above theorem proves a stronger result:

\begin{cora}  Let $\op{P}$ be a Koszul operad and a regular operad and let $A$ be a $\op{P}_\infty$-algebra.  Then there is a chain model for $\op{D}_2$ which acts on the operadic cochains $C^\ast_\op{P}(A)$ inducing the Gerstenhaber structure on cohomology. 
\end{cora}

A proof of Theorem A appeared in the author's PhD thesis \cite{Wardt}.  Let us underscore that with the work of \cite{KS} and \cite{KSch} to lean on, the proof of Theorem A simply boils down to the observation that the minimal operad acts in this more general context.

Cyclic operads were introduced by Getzler and Kapranov in \cite{GeK1} where cyclic homology of algebras over cyclic operads was introduced.  Our construction of the cyclic cohomology of an operad relative to a MC element is a basic generalization of their construction.  Several tools that we develop in the category of cyclic operads are generalizations of the underlying operadic constructions, many of which are due to Ginzburg and Kapranov \cite{GK}.

Theorem B is a generalization of the cyclic Deligne conjecture proved in \cite{KCyclic} using Cacti \cite{Vor},\cite{KCacti}. Alternatively, Tradler and Zeinalian \cite{TZ2},\cite{TZ1} studied the action of Sullivan chord diagrams.  The proof of Theorem B uses the construction of a chain model for the framed little disks given in the author's previous work \cite{Ward}.  This chain model can be considered an $A_\infty$ version of Cacti.

Gravity algebras were introduced by Getzler in \cite{Geteq} and subsequently studied by Westerland \cite{West}.  Theorem C seems to have few precursors on the cochain level.  On the level of homology, inspiration was Getzler's work \cite{G94}, \cite{Geteq} and Chas and Sullivan's string topology \cite{CS}.  However, we make no attempt to address questions of intersection pairing or Poincare duality at the chain level.  Indeed our approach is to work in a setting where these issues do not arise.

From both Deligne's conjecture and algebraic models for string topology stems a large literature of Hochschild cohomology operations and the above list of references is nowhere near exhaustive.  Our approach is to highlight the role of the brace and cyclic brace operations, the associated Lie structures and the Maurer-Cartan formalism.  This has several advantages.  First it is an approach amenable to generalization; moving beyond (non-$\Sigma$) operads to generalizations of operads by identifying the respective Lie structures in each context.  Second, it expands what our chain models naturally act on, essentially viewing the Hochschild complex as simply an example of an operad.  Third, our approach teases out a hierarchy amongst these operations.  For example, we prove that the cyclic bracket is defined before considering the cup product and for us the defining relation of \cite{CS} becomes a property.

Our work here-in is related to the work of Dolgushev and Willwacher \cite{VW} in several ways.  Firstly, an alternate proof of Theorem A  can be extracted from this paper using Willwacher's twisting construction Tw and its fundamental property.  Secondly, our construction of the natural operations in the cyclic setting is a small model for $\text{Tw}$ applied to the cyclic brace operad $\cB$.  As an application of the homology calculation given in the proof of Theorem C we are thus able to compute the cohomology of $\text{Tw}(\cB)$ as a corollary:
\begin{cora}
There is an isomorphism of reduced operads $H^\ast(\text{Tw}(\cB))\cong \Sigma H_\ast(\op{M}_{\ast+1})$.
\end{cora}
To prove this result we use a straight forward adaptation of the arguments of \cite{VW} to the cyclic setting. 

Finally we mention that this work and many of the above references have been inspired directly and indirectly by introduction of Kontsevich's graph complexes \cite{Kont2}, \cite{Kont} and variants \cite{CV},\cite{Will1}.  Our complexes $\nsl{O}$ and $\cnsl{O}$ play the role of graph complexes, and our results could fairly be described as the calculation of natural operations on graph complexes.

\subsection*{Future Directions}  There are many interesting questions left unaddressed in this paper, several of which we would like to highlight.

{\bf  Higher genus analog.}  One may understand the operad structure of $\Sigma H_\ast(\op{M}_{0,\ast+1})$ via the Kimura, Stasheff, Voronov compactification $\underline{\op{M}}_{0,\ast+1}$ of \cite{KSV}, as gluing with a twist parameter.  In higher genus, this parametrized gluing is the topological analog of the $\fr{K}$-twisted modular operads of \cite{GeK2}.  A higher genus analog of the results here-in would seek to consider the action of a chain model for $\op{M}_{g,n}$ on the complex of invariants associated to a (non-$\Sigma$) $\fr{K}$-twisted modular operad and a solution to the quantum master equation.

{\bf  Feynman categorical generalizations.}  More generally, it would be interesting to expand the line of inquiry here-in to objects with operations parametrized by other classes of graphs.  In addition to modular operads one could consider dioperads or properads and the associated cobracket.  In \cite{KW} a generalization of operads is given which permits the construction of a space of natural operations via a nested limit-colimit formula.  This space of natural operations is typically a Lie algebra and one can ask for a general theorem about the space of dg operations after adding a MC element.

{\bf  Deformation theoretic interpretation.}  The bracket on operadic cohomology controls the deformation theory of a given algebra.  In an analogous way the cyclic bracket controls the deformation theory of the algebra under the restriction that deformations and equivalences be inner product preserving.  I have not found a full accounting of this in the literature; see however \cite{PS} in the associative case.  One could also ask for a deformation theoretic interpretation of the higher brackets.

{\bf Cyclic formality of cochains and higher brackets.}  From Example $\ref{polycycex}$ below we see that the complex computing the polydifferential cyclic cohomology associated to a manifold $M$ with volume form $\omega$, is an algebra over $\ZA$.  A cyclic version of the Hochschild-Kostant-Rosenberg theorem says that this complex is quasi-isomorphic to the complex of `$\Delta$-equivariant polyvector fields',  $(T_{poly}^{\ast+1}(M)[u], u\text{div}_\omega)$.  The $\Delta$-equivariant polyvector fields are naturally a dg gravity algebra by restricting the BV operations to only those operations with pass to $\Delta=\text{div}_\omega$ homology.  It is then natural to ask if the $L_\infty$ quasi-isomorphism of Calaque and Willwacher \cite{WillC} can be extended to the higher brackets.

{\bf  Cyclic Deligne conjecture for Koszul Calabi-Yau algebras.}  There is a BV--gravity structure on the Hochschild cohomology -- negative cyclic cohomology of a Koszul Calabi-Yau algebra, see Example $\ref{cyex}$.  How are these structures induced at the chain level?  Note for the underlying Gerstenhaber structure, the usual Deligne conjecture still applies.  But for the BV and gravity structures one expects new chain models acting on the cochains.  This problems should be called the `cyclic Deligne conjecture: B-side'.

{\bf  $E_\infty$ algebras and $S^1$ spaces.}  Given an $S^1$-space, our construction encodes operations on the cohomology and the $S^1$-equivariant cohomology.  However we stop at the $E_2$ level.  The full $E_\infty$ structure is encoded combinatorially by step diagrams, see \cite{MS}.  Step diagrams admit a filtration by $E_n$ operads, and the $E_2$ action coincides with the action described below.  Pushing out the $S^1$ operations and the $E_\infty$ operations along the $E_2$ operations gives an interesting invariant of the space along with the circle action encoded by `step diagrams with spines' which merits further study.  In the equivariant context, we could consider `cylindrical step diagrams', of which our construction would be a sub-operad.

{\bf Comparison of BV/Gravity structures (via symplectic cohomology).}  A general expectation says that (under suitable conditions and qualifiers) the cohomology level BV--gravity operations and the associated long exact sequence of string topology coincide with those on the Hochschild and cyclic cohomology of the Fukaya category of the cotangent bundle.  This correspondence may be achieved via comparisons with symplectic cohomology.  See eg \cite{PaulS} and \cite{MA} and the references there-in.  See \cite{BO1} in the equivariant context.  Both instances of these structures can be seen as arising from our chain level construction, see Example $\ref{fcex}$ and Example $\ref{kiex}$.  As such one may refine this expectation to an equivalence of $\op{TS}_\infty$ (and hence $\op{BV}_\infty)$ algebras in the non-equivariant case and an equivalence of $\ZA$-algebras in the equivariant case.  One may further refine this expectation by considering the possibility of a weak equivalence between the cyclic operads themselves.

\subsection*{Outline}  We begin with a review of the brace operations, operadic Lie algebras, operadic cohomology, the generalized Deligne conjecture, and the minimal operad of Kontsevich and Soibelman is Section $\ref{opssection}$.  This allows in particular for the proof of Theorem A.  In Section $\ref{sec:algcyc}$ we give several fundamental constructions related to cyclic operads.  In particular in this section we define the cyclic brace operad and the associated Lie structures, construct the long exact sequence relating them, define cyclic cohomology of a cyclic operad, establish the model structure on cyclic operads after \cite{KW}, and prove representation theorems for the MC functor.  In Sections $\ref{sec:bv}$ and $\ref{sec:grav}$ we give the actions on the respective complexes and prove Theorem B and Theorem C respectively.  Finally in Section $\ref{exsec}$ we gather together examples of our constructions.  Terminology and conventions for graphs are recorded in Appendix $\ref{sec:trees}$.  A review of the relevant particulars of cyclic operads is contained in Appendix $\ref{sec:cyc}$.

\subsection*{Acknowledgements}
I would like to thank Ralph Kaufmann whose help and guidance laid the groundwork for this project.  I would also like to thank Kei Irie, Dennis Sullivan, Bruno Vallette, Sasha Voronov, and Craig Westerland for several helpful conversations and Dan Petersen for a helpful email.  I am also grateful to the organizers and participants of the Oberwolfach mini-workshop on BV Algebras, Operads, and Hopf Algebroids for many interesting discussions which have improved this paper.  Finally, I would like to gratefully acknowledge the support of the Simons Center for Geometry and Physics where this paper was written.

\tableofcontents

\section{Operads and the generalized Deligne conjecture}\label{opssection}
Let us start by fixing some preliminaries.  We work principally in the symmetric monoidal category of differential graded vector spaces over a field $\kk$ of characteristic zero, although this last assumption is not always necessary.  Given graded vector spaces $V_n$ the vector space $\bigoplus_{n\in \mathbb{N}} V_n$ takes the total grading whereas $\coprod_{n\in \mathbb{N}} V_n$ takes the internal grading.  Our differentials take cohomological grading conventions, keeping in mind the following remark.

\begin{remark}\label{treedegremark}  A CW interpretation of the $\op{A}_\infty$ operad requires homological grading.  More generally we will consider operads with cells indexed by trees whose edges we view to have degree $1$ to match the topology.  When working with cellular operads we define the differentials to have degree $-1$ with the tacit assumption that taking the opposite grading yields an object in our standing category.
\end{remark}

We write $\Sigma V$ (resp. $\Sigma^{-1}V$) for the graded vector space with degrees shifted up (resp. down) $1$ degree from that of $V$.  Recall that by definition $V$ is an odd Lie algebra if $\Sigma^{-1} V$ is a graded Lie algebra.  In particular, the bracket in an odd Lie algebra has degree $-1$.

We assume the reader is familiar with operads; standard references for this material include \cite{MSS}, \cite{LV}.  For now we consider our operads to have $\op{O}(0)=0$.  Thus we are considering `reduced operads' in the parlance of some authors.  Our results can be phrased in both the categories of operads and of non-$\Sigma$ operads.  Of course these categories are related by the forgetful-free adjunction  (where $-\tensor \op{A}s$  is the left adjoint).  Since retaining the symmetric group action will be convenient for suspension and keeping track of signs, we prefer to work in the category of (symmetric) operads, and hence we define:

\begin{definition}\cite{ValD}\label{regdef}  An operad is called regular if it is in the image of the left adjoint $-\tensor \op{A}s$.
\end{definition}

The operadic suspension (resp. desuspension) will be denoted $\fr{s}\op{O}$ (resp. $\fr{s}^{-1}\op{O}$).  Explicitly $\fr{s}\op{O}(n)=\Sigma^{n-1}sgn_n\tensor \op{O}(n)$ where $sgn_n$ is the alternating representation of the symmetric group $S_n$.  The relationship between the suspension of an operad and the suspension of an algebra is the following (see \cite{MSS} Lemma 3.16).

\begin{lemma}\label{MSSlemma}  There is an isomorphism of operads $\fr{s}^{-1}End_A\cong End_{\Sigma A}$.  In particular, $\Sigma A$ is an $\op{O}$-algebra if and only if $A$ is an $\fr{s}\op{O}$-algebra.
\end{lemma}

\begin{remark}\label{signsremark1}
Given an operad $\op{O}$ with structure maps denoted by $\circ_i$ we can identify the elements of $\op{O}(n)$ with the elements of $\fr{s}\op{O}(n)$.  Under this identification, the structure maps $\widetilde{\circ}_i$ of $\fr{s}\op{O}$ satisfy
\begin{equation*}
a\tilde{\circ}_ib = (-1)^{(i-1)(m-1)+(n-1)\text{deg}(b)} a\circ_ib
\end{equation*}
for $a\in \op{O}(n)$ and $b\in\op{O}(m)$.
\end{remark}

\subsection{Algebraic Structure associated to an Operad}  Let $\op{O}$ be a dg operad.  The purpose of this section is to define and study four spaces associated to $\op{O}$ which will be denoted $\nsc{O}, \syc{O}, \nsl{O}, \syl{O}$, each of which is in particular an odd Lie algebra.  

The algebraic structures which we consider here-in arise from odd gluings \cite{KWZ}.  In general odd gluings are not equivalent to standard gluings.  However, operads and cyclic operads are equivalent, through shifts and suspensions, to their odd counterparts, and so we will be able to keep this structure mostly implicit.  Let us however describe the odd structures at work here, so as to give some explanation of the appearance of these shifts and suspensions.  See \cite{KWZ} for the full treatment.

\begin{definition} \label{odddef}  Let $\op{O}$ be a dg $\mathbb{S}$-module such that $\Sigma^{-1}\fr{s}^{-1}\op{O}$ is an operad.  Then we say $\op{O}$ is an odd operad.
\end{definition}

The functor $\Sigma\fr{s}$ is part of an equivalence of categories between odd operads and operads.  A fundamental example is the following.

\begin{example}\label{hochex}  Let $A$ be an associative algebra.  Then the Hochschild cochains of $A$ form an odd operad.  In particular
\begin{eqnarray*}
CH^\ast(A,A)\supset CH^n(A,A) :=\Sigma Hom(A^{\tensor n}, A)\tensor \Sigma^{n-1}\text{sgn}_n = \Sigma \fr{s}End_A(n)
\end{eqnarray*}
so $\{CH^n(A,A)\}_n$ is the oddification of $End_A$.
\end{example}

This example may be unsettling to the reader expecting the definition `$CH^n(A,A) = Hom(A^{\tensor n}, A)$'.  Indeed this definition is perfectly sufficient for defining $CH^\ast(A,A)$ if one is willing to define the differential by hand.  However, here, we want the differential to exist intrinsically.  Making the gluings odd ensures this, as will be seen below.

In analogy with this example, we now define the dg vector spaces associated to $\op{O}$ which we will consider.

\begin{definition}\label{lieopdef}
\begin{eqnarray*}
\nsc{O}:= \coprod_n\Sigma\fr{s}\op{O}(n) & \ \ \ & \syc{O}:= \coprod_n\Sigma(\fr{s}\op{O}(n))_{S_n} \\
\nsl{O}:= \prod_n\Sigma\fr{s}\op{O}(n)& \ \ \ & \syl{O}:= \prod_n\Sigma(\fr{s}\op{O}(n))^{S_n} 
\end{eqnarray*}
\end{definition}

\begin{remark}\label{invrmk}  Since we are working in a field of characteristic zero, the invariants and coinvariants are isomorphic.  We will often make use of this isomorphism and the associated maps:
\begin{equation*}
\xymatrix{\op{O}(n) \ar[r]^{[-] \ }  \ar@/_1pc/[rr]_{\pi}&\op{O}(n)_{S_n} \ar[r]^{\cong \ } & \op{O}(n)^{S_n} }
\end{equation*}
where $\pi(a):=\sum_{S_n} \sigma(a)/n!$.  This isomorphism will often be implicit, and we use the generic term `(co)invariants'.
\end{remark}

A consequence of Lemma $\ref{MSSlemma}$ is that the spaces defined in Definition $\ref{lieopdef}$ are in particular $\op{O}$-algebras.  However this structure does not (always) survive the twisting by a MC element, and is not the structure we will primarily consider.  Rather, in the remainder of this section we will establish odd pre-Lie structures on these spaces as well as brace algebra structures in the non-symmetric cases, $\nsc{O}$ and $\nsl{O}$.  In order to do this we will first recall these structures in the case of the coproducts $\nsc{O}$ and $\syc{O}$ and argue that they extend in a natural way to the products $\nsl{O}$ and $\syl{O}$.

\subsubsection{The pre-Lie structure}
To begin we consider $\coprod_n \op{O}(n)$.  Define $a\circ b =\sum_{i=1}^m a\circ_i b$, for $a\in\op{O}(m)$ and $b\in\op{O}(l)$ and extend $\circ$ linearly to all of $\coprod_n \op{O}(n)$.

\begin{lemma}  $(\coprod_n \op{O}(n), \circ)$ is a dg pre-Lie algebra.
\end{lemma}
Applying the above Lemma to the operad $\fr{s}\op{O}$ makes $\nsc{O}$ an odd dg pre-Lie algebra.  As in Remark $\ref{signsremark1}$, we will often identify the elements of $\coprod_n \op{O}(n)$ with the elements of $\nsc{O}$, in which case the pre-Lie operation in $\nsc{O}$ is given by
\begin{equation}\label{signsline}
a\circ b =\ds\sum_{i=1}^n (-1)^{(i-1)(m-1)+(n-1)deg(b)}  a\circ_i b
\end{equation}
for $a\in\op{O}(n)$ and $b\in\op{O}(m)$, where $\circ_i$ denotes the structure maps of $\op{O}$. 

 Since $\nsc{O}$ is an odd dg pre-Lie algebra, its odd graded commutator is an odd Lie bracket:
\begin{equation*}
[a,b] := a\circ b - (-1)^{(|a|-1)(|b|-1)}b\circ a
\end{equation*}
The construction of the odd dg Lie algebra $(\nsc{O}, [-,-], d)$ is due in its original form to Gerstenhaber \cite{G} for the operad $\op{O}=End_A$ for an associative algebra $A$, in which case $\nsc{O}=CH^\ast(A,A)$, as in Example $\ref{hochex}$.

\subsubsection{Brace operations}  The odd pre-Lie operation defined above is the first in a family of so-called brace operations, which we now define.  These operations were first described by Getzler \cite{Brace1} and Kadeishvili \cite{Brace2} in the Hochschild context and by Gerstenhaber and Voronov \cite{GV} in a general operadic context.

\begin{definition}  Let $a, b_1\cdc b_n$ be arbitrary elements of an operad $\op{O}$ of arities $r, t_1\cdc t_n$ respectively.  We define the following element of $\coprod\op{O}(n)$: 
\begin{equation}\label{braceeq}
a\{b_1,\dots,b_n\}:= \ds\sum_I (\dots((a\circ_{i_1} b_1)\circ_{i_2}b_2)\dots \circ_{i_n} b_n)
\end{equation}
where the set $I$ consists of $n$-tuples $(i_1\cdc i_n)$ such that $i_{j+1}\geq i_j+t_j$ and $i_n \leq r-n+1+\sum_{l=1}^{n-1}t_l$.  By convention, a sum over the empty set is zero. The (ungraded) brace operation $B^0_n \in End_{\coprod\op{O}(n)}(n+1)$ is defined by:
\begin{equation*}
B^0_n(a;b_1,\dots,b_n)= a\{b_1,\dots,b_n\}
\end{equation*}
In particular $B^0_1(a;b)=a\circ b$.
\end{definition}

The suboperad of $End_{\coprod\op{O}(n)}$ generated by the ungraded brace operations has a useful interpretation in terms of planar rooted trees which we now define.

\begin{definition}\label{braceopdef}
Let $\B^0(n)$ be the $S_n$-module spanned by labeled planar rooted trees (defined in Appendix $\ref{sec:trees}$) with $n$ vertices.  Define an operad structure on the $\mathbb{S}$-module $\B^0$ by
\begin{equation*}
T\circ_i T^\prime := \sum_{[T,T^\prime,i]} T^{\prime\prime}
\end{equation*}
where $[T,T^\prime, i]$ is the set of labeled planar rooted trees such that $T^{\prime\prime}\in [T,T^\prime, i]$ if and only if both
\begin{itemize}
\item  The full subtree  of $T^{\prime\prime}$ generated by the vertices $i\cdc i+m-1$ is isomorphic to $T^\prime$ and,
\item  Under the identification above, $ T^{\prime\prime}/T^{\prime}\cong T$.
\end{itemize}
\end{definition}

Let $\beta^0_n$ be the planar rooted tree of height two with $n+1$ vertices, having one vertex of height $1$ labeled by $1$ and the remaining $n$ vertices of height $2$ labeled in the planar order.

\begin{lemma}  The trees $\{\beta^0_n\}$ generate the operad $\B^0$ under the operadic composition and $S_n$-action.  The assignment $\beta^0_n\mapsto B^0_n$ induces an injective morphism of operads $\B^0\to End_{\coprod\op{O}(n)}$. \end{lemma}

The $0$ in the notation is meant to remind one that $\B^0$ is just a linear operad, i.e. concentrated in degree zero.  We will more often consider the desuspension of $\B^0$, which has the effect of giving edges degree $-1$.
\begin{definition}\label{bracedef}  Define $\B$ to be the operad $\fr{s}^{-1}\B^0$.  Define $B_n$ to be the image of $B_n^0$ under the desuspension.
The operation $B_n$ will be called the $n^{th}$ brace operation. 
\end{definition}

\begin{lemma}\label{bracelem}  For any operad $\op{O}$ (resp. non-$\Sigma$ operad), $\nsc{O}$ is a $\B$-algebra.  Moreover, if $\op{O}$ is a regular operad then $\syc{O}$ is a $\B$-algebra.
\end{lemma}

In accordance with Remark $\ref{signsremark1}$, the $\B$-algebra structure on $\nsc{O}$ is still denoted $B_n(a;b_1,\dots,b_n)= a\{b_1,\dots,b_n\}$ where now the $\circ_i$ compositions appearing in the sum $a\{b_1,\dots,b_n\}$ are those in the operad $\fr{s}\op{O}$.

\begin{remark}\label{symrmk}  In the symmetric setting, the space $\syc{O}$ is also a pre-Lie algebra, whose product $\overline{\circ}$ is given by the formula $[a]\overline{\circ}[b]:=[a\circ b]$ \cite{KM}.  There is a symmetric version of the brace operad, but it turns out to be entirely generated by the pre-Lie operation, see \cite{CL} and \cite{LM}.
\end{remark}

\subsubsection{Extendable operations}\label{extsec}  Above we have shown that $\nsc{O}$ is a $\B$-algebra.  We would now like to extend these operations to the product $\nsl{O}$.  In order to do so we now give a condition which permits the extension of an operation on the coproduct to the product.  

Write $A_\bullet:=\coprod_{i\in \mathbb{N}}A_i$ and $A^\bullet:=\prod_{i\in \mathbb{N}}A_i$ for a collection of vector spaces $A_i$.

\begin{definition}  An operation $\phi\in End_{A_\bullet}(n)$ is called extendable if for every $N\in \mathbb{N}$, $\pi_N\circ\phi(A_{r_1}\tdt A_{r_n})=0$ for all but finitely many $n$-tuples $(r_1\cdc r_n)$, where $\pi_N\colon \coprod A_i \to A_N$ is the projection.
\end{definition}

\begin{lemma}  The extendable operations form a suboperad of $End_{A_\bullet}$, which we denote by $End^{ex}_{A_\bullet}$.  Moreover there is a morphisms of operads
\begin{equation*}
End_{A_\bullet}\supset End^{ex}_{A_\bullet} \to End_{A^\bullet}
\end{equation*}
\end{lemma}
\begin{proof}  To show that $End^{ex}_{A_\bullet}$ is a suboperad is straight-forward.  For the second claim, if $\phi$ is an extendable operation of arity $n$ then for a fixed $N$ there is a finite sets of $n$-tuples, call this set $J_{N,\phi}$, such that the projection is non-zero.  Then $\phi$ acts on $A^\bullet$ in the $N$ factor via the composition,
\begin{equation}
(A^\bullet)^{\tensor n} \to \ds\coprod_{J_{N,\phi}} A_{j_1}\tdt A_{j_n} \stackrel{\pi_N\circ \phi}\to A_N
\end{equation}
and we appeal to the universality of the limit to define a map to $A^\bullet$.
\end{proof}

\begin{proposition}  The brace operations are extendable.
\end{proposition}
\begin{proof}  For a generator $\beta_m$ we see that $\pi_N\beta_m(\op{O}(n_0)\tdt \op{O}(n_m))=0$ unless $N=-m+\sum n_i$.
\end{proof}
\begin{corollary}  $\nsl{O}$ is a $\B$-algebra and $\syl{O}$ is a pre-Lie algebra.
\end{corollary}

In particular the brace algebra structure on $\nsl{O}$ is determined in a natural way by the action on individual factors. 

\subsection{Algebraic Structure associated to an Operad + MC element}  

We will be interested in the above odd Lie algebras along with choices of Maurer-Cartan (MC) elements.  Given an odd dg Lie algebra $(\fr{g},[-,-],d)$ we define the Maurer-Cartan set of $\fr{g}$, called $MC(\fr{g})$ to be the elements $\zeta\in\fr{g}$ of degree $2$ such that
\begin{equation*}
0=d(\zeta)+\frac{1}{2}[\zeta,\zeta]
\end{equation*}
We will often view $MC(-)$ as a functor from the category of dg Lie algebras to the category of sets.  Note the fact that our Lie algebra is odd results in MC elements of degree $2$, instead of the usual degree $1$.

\begin{lemma}\label{mclemma} \label{twistedder}  Let $\fr{g}$ be an odd dg Lie algebra with $\zeta\in MC(\fr{g})$.  Then the equation
\begin{equation*}
\delta_\zeta(-):=d(-)+[\zeta,-]
\end{equation*}
defines a square zero differential on $\fr{g}$ of degree $+1$ which is an odd derivation of $[-,-]$.
\end{lemma}

Finally notice that if $\fr{g}$ is in particular odd dg pre-Lie, the MC equation becomes $0=d(\zeta)+\zeta\circ\zeta$.

\subsubsection{The operad case}  Given an operad $\op{O}$ and a MC element $\eta$ in one of the odd Lie algebras $\nsc{O}, \nsl{O}, \syc{O}, \syl{O}$ we will often consider said Lie algebra with the twisted differential $\delta:=\delta_\eta$ as constructed above.  In this case if we write $\eta=\prod \eta_n$, then we must have $\text{deg}(\eta_n)=2-n$.  In particular $\eta_1\in\op{O}(1)$ has internal degree $1$, and by arity considerations $d_{\eta_1}(a):=d_{\op{O}}(a)+[\eta_1,a]$ is a differential (since $\delta$ is).  As such we may always restrict our attention to MC elements of the form $\eta^\prime:=\prod_{n\geq 2}\eta_n$ at the cost of replacing the original aritywise differential with $d_{\eta_1}$.

\subsubsection{Representability of MC elements.}
Recall (e.g. from \cite{GK}) the operads $\op{A}_\infty$ encoding $A_\infty$ algebras and $\op{L}_\infty$ encoding $L_\infty$ algebras.  The following well known theorem states that these objects represent the MC functors.

\begin{theorem}\label{repthm}  There are natural bijective correspondences
\begin{equation*}
MC(\nsl{O})\cong Hom_{dg\op{O}ps}(\op{A}_\infty, \op{O}) \ \ \ \ \text{and} \ \ \ \  MC(\syl{O})\cong Hom_{dg\op{O}ps}(\op{L}_\infty, \op{O})
\end{equation*}
\end{theorem}

We may occasionally abuse notation by using the same character to refer to both sides of this correspondence.

\begin{remark}  Theorem $\ref{repthm}$ tells us that the category of pairs $(\op{O}, \eta)$ where $\op{O}$ is an operad and $\eta$ is a MC element is an undercategory, and thus a model category where the canonical forgetful functor creates weak equivalences, fibrations, and cofibrations \cite{hirsch}.
\end{remark}

In the event that $\op{O}$ is not dg, viewed as dg with trivial differential and concentrated in degree zero, the situation simplifies as follows.

\begin{corollary}\label{MCcor}  For a linear operad $\op{O}$ (not dg) there are natural bijective correspondences
\begin{equation}
MC(\nsc{O})\cong Hom_{\op{O}ps}(\op{A}s, \op{O}) \ \ \ \ \text{and} \ \ \ \  MC(\syc{O})\cong Hom_{\op{O}ps}(\op{L}ie, \op{O})
\end{equation}
\end{corollary}

There is a standard commutator map $\op{L}_\infty\to\op{A}_\infty$ and thus we have a map of sets $MC(\nsl{O})\to MC(\syl{O})$.  It is natural to ask when there is a lift and this occurs when $\op{O}$ is regular.  Moreover we have:
\begin{lemma}\label{liftlem}  If $\op{O}$ is regular then the map $MC(\nsl{O})\to MC(\syl{O})$ is a bijection.
\end{lemma}

In the non dg case, the lifting interpretation applies with the standard (commutator) map $\op{L}ie\to\op{A}s$.

\subsection{Operadic cohomology theories: a fundamental example}\label{opcosec}  Let $\op{P}$ be a finitely generated Koszul operad and let $A$ be a $\op{P}_\infty$-algebra.  There is a sequence $\op{L}ie\to\op{P}^!\circ\op{P}\to \op{P}^!\tensor\op{P}$ (\cite{GK} Corollary 2.2.9b, see also Lemma $\ref{cll}$ below), where $\circ$ represents the Manin white product and where $\op{P}^!$ is the quadratic dual of $\op{P}$.  Since $\op{L}_\infty$ is cofibrant this morphism lifts to a morphism $\op{L}_\infty\to D(\op{P})\tensor\op{P}$.\footnote{Here $D$ is the `dual dg operad' in the sense of \cite{GK}.  That is, $D(\op{P})=\Omega(\op{P}^\ast)$ in the notation of $eg$ \cite{LV}.}  Thus we have a sequence of dg operads
\begin{equation}\label{mccoelement}
\op{L}_\infty\to \op{P}^!\tensor D(\op{P}^!)\to \op{P}^!\tensor End_A
\end{equation}
The composite morphism allows us to define the $\op{P}$-cochains of a $\op{P}_\infty$-algebra.

\begin{definition}  Let $\op{P}$ be a Koszul operad as above and let $A$ be a $\op{P}_\infty$-algebra.  Define $\op{O}=\op{P}^!\tensor End_A$.  Then the completed $\op{P}$-cochains of $A$ are defined to be the cochain complex,
\begin{equation*}
\hat{C}^\ast_\op{P}(A):=(\syl{O}, \delta_\eta)
\end{equation*}
where $\delta_\eta:=d+[\eta,-]$ after Lemma $\ref{mclemma}$ with MC element from equation $\ref{mccoelement}$ via Theorem $\ref{repthm}$.  The completed $\op{P}$-cohomology of $A$, denoted $\hat{H}^\ast_\op{P}(A)$, is defined to be the cohomology of this complex.

In the case that the MC element $\eta\in MC(\syl{O})$ lifts to ${\eta^\prime}\in MC(\syc{O})$ along the standard inclusion $\syc{O}\to\syl{O}$, we define the (noncompleted) cochain complex
\begin{equation*}
C^\ast_\op{P}(A):=(\syc{O}, \delta_{\eta^{\prime}})
\end{equation*}
and define $H^\ast_\op{P}(A)$ to be the cohomology of this complex.
\end{definition}

When $\op{P}=\op{A}s$ we recover Hochschild complex of an associative or $A_\infty$ algebra.  When $\op{P}=\op{L}ie$ we recover the Chevalley-Eilenberg complex of a Lie or $L_\infty$ algebra.  Further examples of interest include Harrison cohomology \cite{Harr}, Poisson cohomology \cite{P1}, Leibniz cohomology \cite{LL}, etc.  These cohomology theories provide a framework to extend the classical deformation theory of Gerstenhaber \cite{Gerst2} and Nijenhuis and Richardson \cite{NR} in the associative and Lie contexts.  See \cite{GerstS} and \cite{LV}.

\subsection{The minimal operad}\label{minopsec}
In this section we will define an operad $\TA$ with the following property:  given any dg operad $\op{O}$ and an element $\zeta\in MC(\nsl{O})$, the operad $\TA$ acts on $(\nsl{O},\delta_\zeta)$.  The action will be the subject of Section $\ref{actionsection}$.  The operad $\TA$ is isomorphic to the `minimal operad' of \cite{KS} and is a chain model for the little disks operad $\op{D}_2$.  As an informal description the operad $\TA$ is an insertion operad of rooted $A_\infty$ labeled trees.  When reading this section note that our terminology and conventions for trees are recorded in Appendix $\ref{sec:trees}$.

\begin{definition}  Define $\TA(n)$ to be the graded vector space generated by rooted $A_\infty$-labeled trees with $n$ white vertices with $S_n$ action by a signed permutation of the labels of the white vertices. 
\end{definition}

For a rooted $A_\infty$-labeled tree $T$ we define the grading and the differential $\partial$ `locally'.  The degree of a white vertex is one less than the number of arcs, and the degree of a black vertex is the degree of the associahedron cell which labels it.  The degree of a tree is the sum of the degrees of its vertices, keeping in mind Remark $\ref{treedegremark}$.  To define the differential we first define the differential at a vertex $v$, call this $\partial(T;v)$, and then define
\begin{equation}\label{mindif}
\partial(T):=\sum_{v\in T}\pm \partial(T;v)
\end{equation}
To define $\partial(T;v)$ we have two cases.

$\textbf{Case 1}$: $v$ is a black vertex with label $\alpha\in\op{A}_\infty(m)$.  In this case we define $\partial(T;v)$ to be the tree resulting from relabeling vertex $\alpha$ with $d(\alpha)\in\op{A}_\infty(m)$.

$\textbf{Case 2}$: $v$ is a white vertex.  In this case $\partial(T;v)$ is a sum of all trees which can be formed by contracting one or more adjacent white angles (see $\ref{g9}$).

The operad structure of $\TA$ is the same as the operad structure of the brace operations, if one ignores the vertex coloring.  Notice that this includes grafting of branches on to black vertices by increasing the arity of the label (see Appendix $\ref{sec:trees}$).

To fix the signs in the above discussion it is convenient to realize the operad $\TA$ as the cellular chains of a topological (quasi)-operad,  see \cite{KSch}.  Then a choice of orientation of the cells fixes the signs in the $S_n$ action, the composition operations and the differential.  This approach also makes clear the fact that $\partial^2=0$.  There are several ways to choose such an orientation, each of which is natural, see subsection 3.2 of \cite{KDeligne} for details.  In particular, using the terminology of Appendix $\ref{sec:trees}$, an orientation of a cell corresponding to a rooted tree $T$ is specified by an order (mod 2) of the union of the set of white edges of $T$ with the black vertices of $T$.  We take as our convention the order induced by the embedding of $T$ into the plane, starting at the root, where an edge or vertex is recorded at first contact.  This choice agrees with \cite{KSch} and the order `Nat' of \cite{KDeligne}.

\begin{definition}\label{simplegluingdef}  Let $T_1$, $T_2$ be trees in $\TA$.  A composition $T_1\circ_iT_2$ is called simple if the vertex labeled by $i$ is of maximum height.  Notice that by definition the result of a simple gluing is a single rooted tree.
\end{definition}  

The importance of the simple gluings is the content of the next lemma.

\begin{lemma}  Any $A_\infty$-labeled tree in $\TA$ can be formed via simple gluings of the corollas and brace operations $\beta_n$ along with the $S_n$ action.
\end{lemma}

In light of this lemma we refer to brace operations and corollas as the generators of $\TA$.

\subsection{The action of the minimal operad.}\label{actionsection}  Let $\op{O}$ be a dg operad and let $\zeta\in MC(\nsl{O})$.  In this subsection we will define a morphism of dg operads $\rho_\zeta=\rho$, 
\begin{equation*}
\rho\colon\TA\to End_{(\nsl{O},\delta_\zeta)}
\end{equation*}
We start by defining the action of the generators:
\begin{itemize}
\item  \textbf{Corollas:}  The unique generator with one white vertex maps to the identity operation.  A corolla whose lone black vertex is labeled by $\mu_n$ for $n \geq 2$ is mapped under $\rho$ to,
\begin{equation*} 
\rho(\mu_n)(a_1\cdc a_n)=\ds\prod_{s\geq n}\zeta_s\{a_1\cdc a_n\}
\end{equation*}
Given the convention that the sum over an empty set is zero we can equivalently write
\begin{equation*} 
\rho(\mu_n)(a_1\cdc a_n)=\zeta\{a_1\cdc a_n\}
\end{equation*}
\item  \textbf{Braces:} We define $\rho(\beta_n)=B_n$.  In the bracket notation we may write
\begin{equation*}
\rho(\beta_n)(a,b_1,\dots, b_n)=a\{b_1,\dots, b_n\}
\end{equation*}
\end{itemize}

\begin{theorem}\label{actionthm1}  The above assignment extends to a morphism of dg operads $\rho\colon\TA\to End_{(\nsl{O},\delta_\zeta)}$.  In particular, $(\nsl{O},\delta_\zeta)$ is an $\TA$-algebra.
\end{theorem}

\begin{proof}
Let us simply outline the steps of a proof which emphasizes the fundamental role played by the generators, i.e. the braces and corollas.  See \cite{Wardt} for a finer level of detail.

{\bf Step 1:} extend $\rho$ via simple gluings and show this extension is independent of choice of decomposition.

Let $T$ be any tree appearing in $\TA$ and choose a decomposition into braces and corollas, such that each composition is simple.  Then define $\rho(T)$ to be the composition of $\rho$ of the generators.  Since each composition occurs at a vertex of maximum height, such a decomposition corresponds to a decomposition of a tree (forgetting the extra data), and so the operad associativity of $End_{\nsl{O}}$ ensures that $\rho(T)$ is well defined independent of the choice of such a decomposition.

{\bf Step 2:}  show that $\rho$ respects the composition of generators.

If the composition happens to be simple then this is true by definition, so we can restrict our attention to the case of a non-simple composition of generators.  There are thus two cases $\beta_n\circ_1 \beta_m$ and $\beta_n\circ_1 \mu_m$.  That the former holds is a consequence of the $\B$-algebra structure on $\nsl{O}$ established in Lemma $\ref{bracelem}$.  The latter follows similarly by evaluating at $\zeta$ in the first factor.

{\bf Step 3:}  argue that $\rho$ respects respects all compositions. 

To show $\rho(T\circ_iT^\prime)= \rho(T)\circ_i\rho(T^\prime)$, induct on the number of generators in a decomposition of $T^\prime$ into simple compositions of generators.  For the base case, $T^\prime$ is a generator and we can use operad associativity to rewrite $T\circ_iT^\prime$ as a sequence of compositions each of which is simple or between generators.  The induction step then follows from operad associativity and the induction hypothesis.

{\bf Step 4:}  show that $\rho$ respects the differential on generators.

This entails two lengthy but straight forward computations checking compatibility with $\beta_n$ and $\mu_n$.  We again refer to \cite{Wardt} for details.

{\bf Step 5:}  show $\rho$ respects the differential.

Again use a decomposition of $T$ into generators and the fact that $\TA$ is a dg operad to write
\begin{equation*}
\partial[(\dots(g_1\circ_{i_1}\dots \circ_{i_n}g_{n+1})\dots)] =\sum_{j=1}^n(-1)^{|g_1|+\dots+|g_{j-1}|}(\dots(g_1\circ_{i_1}\dots \circ_{i_{j-1}}\partial(g_j)\circ_{i_j}\dots\circ_{i_n}g_{n+1})\dots)
\end{equation*}
Applying $\rho$ and the above steps yields the desired result. \end{proof}

\begin{corollary}\label{nonsigmaaction}  Let $\op{N}$ be non-$\Sigma$ dg operad and $\zeta\in MC(\nsl{N})$.  Then $(\nsl{N}, \delta_\zeta)$ is an  $\TA$-algebra.  Equivalently, if $\op{O}$ is a regular operad then $\syl{O}$ is an $\TA$-algebra.
\end{corollary}
\begin{proof}  This follows from Lemmas $\ref{bracelem}$ and $\ref{liftlem}$.
\end{proof}

\subsection{Deligne's Conjecture}  In this section we will give our generalization of Deligne's conjecture via the above work.

\begin{definition}\label{smile}  Given $\zeta \in MC(\nsl{O})$ we define a bilinear multiplication $\nsl{O}\tensor\nsl{O}\stackrel{\cupp}\to \nsl{O}$ by, 
\begin{equation*}
-\smile - = B_2(\zeta;-,-)
\end{equation*}
\end{definition}

\begin{lemma}\label{GO} Given a dg operad $\op{O}$ and $\zeta\in MC(\nsl{O})$ the associated $\cupp$-product passes to cohomology, making $(H^\ast(\nsl{O},\delta_\zeta), [-,-], \cupp)$ a Gerstenhaber algebra.  If $\zeta$ lifts to $MC(\nsc{O})$ along the standard inclusion $\nsc{O}\to\nsl{O}$ then $(H^\ast(\nsc{O},\delta_\zeta), [-,-], \cupp)$ is a Gerstenhaber algebra and the induced map is a map of Gerstenhaber algebras.
\end{lemma}
\begin{proof} The first part of this theorem follows from the fact that $(\nsl{O}, \delta_\zeta)$ is a $\TA$-algebra along with the fact that $H_\ast(\TA)$ is the Gerstenhaber operad, as we will discuss below.  The second part follows from the fact that if the MC element lifts then the induced action on $\nsc{O}\subset \nsl{O}$ is closed.
\end{proof}

In the symmetric case we may consider the cohomology structure of $H^\ast(\syl{O}, \delta_\eta)$.  In this case the bracket lifts to the cohomology level, and the $\op{L}_\infty$ structure induces another Lie bracket on cohomology.  However this second Lie bracket is the boundary of the pre-Lie operation, and hence is zero on the level of cohomology.

We continue to write $\op{D}_2$ for the little disks operad.  A theorem of F. Cohen (see \cite{CLM}) says that a graded vector space is an algebra over $H_\ast(\op{D}_2)$ if and only if it is a Gerstenhaber algebra.  Therefore, given any associative algebra $A$, $HH^\ast(A,A)$ is an algebra over $H_\ast(\op{D}_2)$.  The original Deligne conjecture is a chain level version of this statement.  We now give a generalization which replaces $HH^\ast(A,A)$ with $H^\ast(\nsl{O}, \delta_\zeta)$.  See the discussion in the introduction and Remark $\ref{litrmk}$ for connections to the literature.  The proof follows from Theorem $\ref{actionthm1}$ above.

\begin{theorem}\label{Dconj3}  Let $\op{O}$ be a dg operad and let $\zeta\in MC(\nsl{O})$.  There is a chain model for $\op{D}_2$, namely $\TA$, which acts on $(\nsl{O}, \delta_\zeta)$ inducing the Gerstenhaber structure (Lemma $\ref{GO}$) on cohomology.  If $\zeta$ lifts to $MC(\nsc{O})$, then $(\nsc{O}, \delta_\zeta)$ is also an $\TA$-algebra, and the standard inclusion $\nsc{O}\to\nsl{O}$ is a morphism of $\TA$-algebras.
\end{theorem}

\begin{corollary}\label{Dconj4}  Let $\op{P}$ be a Koszul operad and a regular operad, and let $A$ be a $\op{P}_\infty$-algebra.  Then $\TA$ is a chain model for $\op{D}_2$ which acts on the operadic cochains $\hat{C}^\ast_\op{P}(A)$ inducing the Gerstenhaber structure on cohomology.  If $A$ is in particular a $\op{P}$-algebra the $\TA$ action restricts to the uncompleted operadic cochains $C^\ast_\op{P}(A)$ inducing the Gerstenhaber structure on cohomology.
\end{corollary}

\begin{proof}  Since $\op{P}$ is a regular operad so is $D(\op{P}^!)$.  Apply Corollary $\ref{nonsigmaaction}$ to $\op{O}=D(\op{P}^!)\tensor End_A$.
\end{proof}

\begin{remark}\label{litrmk} The first statement of Corollary $\ref{Dconj4}$ in the case when $\op{P}=\op{A}s$ is equivalent to the $A_\infty$ Deligne conjecture proven in \cite{KS, KSch}.  The second statement of this corollary when $\op{P}=\op{A}s$ is equivalent to the original Deligne conjecture.  The second statement has also been proven in the case of an associative dialgebra by Majumdar and Mukherjee \cite{MM}. Yau proves the second statement of Corollary $\ref{Dconj4}$ for several examples of Loday algebras \cite{YauD}.  In \cite{ValD}, Vallette proves this result for all such $\op{P}$ and an arbitrary $\op{P}$-algebra.  The $\op{P}_\infty$-algebra case is to my knowledge new.  See the introduction for further discussion and references regarding Deligne's conjecture.  
\end{remark}

We conclude this section by connecting our results to the notion of $\op{G}_\infty$-algebras.

\begin{corollary}  Let $\op{O}$ be a dg operad and let $\zeta\in MC(\nsl{O})$.  Then $(\nsl{O}, \delta_\zeta)$ is a $\op{G}_\infty$-algebra.  In particular for $\op{P}$ a Koszul operad and a regular operad and $A$ a $\op{P}_\infty$-algebra, $\hat{C}^\ast_\op{P}(A)$ is a $\op{G}_\infty$-algebra.
\end{corollary}
\begin{proof}  This follows from a standard model category argument and the fact that $\op{G}_\infty$ is cofibrant, along with formality of the operad $\op{D}_2$.
\end{proof}

\section{Algebraic structure associated to a cyclic operad}  \label{sec:algcyc}

In considering the generalized Deligne conjecture (Theorem $\ref{Dconj3}$) in Section $\ref{opssection}$ we made use of operadic constructions including associated odd Lie algebras, the brace and symmetric brace operations, the representability of MC elements, Manin products, the homotopical algebra of operads and operadic cohomology theories.  When working with cyclic operads there are generalizations of these constructions which may or may not have appeared in the literature.  The purpose of this section is to develop the above constructions together in the language of cyclic operads as a means of developing a cyclic operadic analog of the results of Theorem A.

In this section we will first show that given a cyclic operad $\op{O}$ we can associate to it odd Lie algebras $\cnsc{O}, \cnsl{O}, \csyc{O}, \csyl{O}$, much in analogy with the case of a (non-cyclic) operad above.  These Lie structures are variations of the Lie bracket first given in this generality in \cite{KWZ}, although particular examples preceded that work, (see \cite{BLB},\cite{CV},\cite{MenichiLie}).  We then show, in analogy with Definition $\ref{lieopdef}$, that these Lie brackets are the first in respective families of `brace-like' operations.  We place more emphasis on the non-$\Sigma$ case, in which the operations will be called cyclic brace operations.  Also in this section we prove representability theorems for the associated Maurer-Cartan functors and construct the long exact sequence.

When reading this section, note that a review of cyclic operads, including associated terminology and notation, is contained in Appendix $\ref{sec:cyc}$.

\begin{definition}  For a cyclic operad $\op{O}$ we define vector spaces:
\begin{eqnarray*}
\cnsc{O}:= \coprod_n\Sigma(\fr{s}\op{O}(n))_{\Z_n^+} & \ \ \ & \csyc{O}:= \coprod_n\Sigma(\fr{s}\op{O}(n))_{S_n^+} \\
\cnsl{O}:= \prod_n\Sigma(\fr{s}\op{O}(n))^{\Z_n^+} & \ \ \ & \csyl{O}:= \prod_n\Sigma(\fr{s}\op{O}(n))^{S_n^+} 
\end{eqnarray*}
\end{definition}

We call $\cnsc{O}$ and $\cnsl{O}$ the non-$\Sigma$ cyclic (co)invariants and we call $\csyc{O}$ and $\csyl{O}$ the full cyclic (co)invariants.  Note that the non-$\Sigma$ cyclic (co)invariants make sense for a cyclic operad and a non-$\Sigma$ cyclic operad.

\subsection{The cyclic bracket}  In this subsection we will define the odd Lie structures on $\cnsc{O},\cnsl{O},\csyc{O},\csyl{O}$ for the cyclic operad $\op{O}$.  The coinvariant form of the bracket uses the $\circ_{ij}$ maps of Definition $\ref{circmaps}$.

\begin{theorem}\label{KWZthm} \cite{KWZ} Let $\op{O}$ be a cyclic operad.  The operation $\sum\circ_{ij}$ induces an odd dg Lie bracket on both $\cnsc{O}$ and $\csyc{O}$ which we call the (coinvariant) cyclic bracket.
\end{theorem}
\begin{proof}  We will denote the Lie bracket by $\{-,-\}$.  To be precise, for $a\in\Sigma\fr{s}\op{O}(n)$ and $b\in\Sigma\fr{s}\op{O}(m)$ we define
\begin{equation}
\{[a],[b]\}:=\sum_{\substack{0 \leq i \leq n \\ 0 \leq j \leq m }} [a\circ_{ij} b]
\end{equation}
where $[-]$ denotes the class under the $\Z_n^+$ (resp. $S_n^+$) action.  That this operation is well defined and odd commutative follows from Lemma $\ref{extlem}$ and the fact that $\Sigma\fr{s}\op{O}$ is an odd cyclic operad.  Verifying the odd Jacobi identity is straight-forward and we refer to \cite{KWZ}. \end{proof}

\begin{corollary} Let $\op{O}$ be a cyclic operad.  The standard inclusion $\cnsl{O}\hookrightarrow \nsl{O}$ is a morphism of Lie algebras.  In particular, the operadic bracket of two invariant elements is invariant.  Specifically, the bracket on invariants can be written as $N(a\circ_1 b)$.  On the full cyclic invariants, the operation $\{a,b\}=\sum_{\sigma\in S_n^+}\frac{1}{n!}\sigma(a\circ_1 b)$ defines an odd Lie bracket on $\csyl{O}$.  We refer to these brackets as the (invariant) cyclic brackets.  
\end{corollary}
\begin{proof}
The Lie bracket defined above on coinvariants is clearly extendable in the sense of subsection $\ref{extsec}$. To prove the corollary we first translate the coinvariant bracket via the standard isomorphism.  For example in the non-symmetric case for invariant elements $a$ and $b$ the coinvariant bracket translates as:
\begin{equation*}
\sum_{i,j}\overline{a\circ_{ij}b}=\frac{(n+1)(m+1)}{n+m} N(a\circ_1 b)
\end{equation*}
where $\overline{a\circ_{ij}b}$ means projection to the cyclic invariants of $a\circ_{ij}b$.  Since multiplication by the degree in an odd Lie bracket produces another odd Lie bracket, it follows that $N(a\circ_1 b)$ is an odd Lie bracket.  Then it is easy to see from the axioms of a cyclic operad (Definition $\ref{cycdef}$) that $N(a\circ_1 b)=[a,b]$ for invariant elements $a$ and $b$.
\end{proof}

Given a cyclic operad, we now have two different Lie algebra constructions each with four distinct variations.  To keep track of these eight odd Lie algebras, we give the following corollary.  The notation is $\pi(a)=\sum_{\sigma\in S_n}\sigma(a)/n!$, $N(a)= \sum_{\sigma\in\mathbb{Z}_n^+}\sigma(a)$, $\bar{N}(a):= \sum_{\sigma\in S_n^+}\sigma(a)/n!$, $iso$ means induced by the levelwise isomorphism, and $inc$ means induced by levelwise inclusion.

\begin{corollary}  Let $\op{O}$ be a cyclic operad.  There are morphisms of odd Lie algebras:
\begin{equation}\label{cube}
\xymatrix{ \cnsc{O} \ar[rr]^{[-]} \ar[dd]_{N} \ar[dr]^{N} & & \csyc{O} \ar@{-}[d]_{\bar{N}} \ar[dr]^N & \\ & \nsc{O} \ar@{-}[d]^{iso} \ar[dd] \ar[rr]^{[-] \ \ \ \ \ \ \ \ \ } & \ar[d] & \syc{O} \ar[dd]^{\pi} \\ \cnsl{O} \ar@{-}[r]^{\pi} \ar[dr]_{inc} & \ar[r] & \csyl{O} \ar[dr]^{inc} \\ & \nsl{O} \ar[rr]^{\pi} && \syl{O} }
\end{equation}
\end{corollary}

This diagram is in the category of odd Lie algebras.  Shifting down, there is a corresponding diagram of honest Lie algebras associated to any anti-cyclic operad.

Finally, let us observe an important distinction between the cyclic brackets and the operadic brackets.  In the operad case the bracket was a commutator of a pre-Lie operation.  In particular, in the non-$\Sigma$ case this pre-Lie operation was part of a hierarchy of chain level operations called the braces.  In the cyclic case the bracket is not a commutator.  However there is a notion of cyclic brace operations in the planar case which we define presently.

\subsection{Cyclic brace operations}  The odd Lie bracket defined above in the planar case is the first in a collection of operations which we call cyclic brace operations.  In this subsection we will introduce the operad generated by these operations, denoted $\cB$, and called the cyclic brace operad.  To be precise, we will consider operads $\cB^+, \cB^-,\cB$ which act on the (co)invariants of a cyclic, anti-cyclic and odd-cyclic operad respectively.  Recall that our terminology for trees and graphs is discussed in Appendix $\ref{sec:trees}$.

\begin{definition}
Let $\cB^+(n)$ be the $S_n$-module spanned by labeled planar trees with $n$ vertices.  Define an operad structure on the $\mathbb{S}$-module $\cB^+$ by
\begin{equation*}
T\circ_i T^\prime := \sum_{[T,T^\prime,i]} T^{\prime\prime}
\end{equation*}
where $[T,T^\prime, i]$ is the set of labeled planar trees such that $T^{\prime\prime}\in [T,T^\prime, i]$ if and only if both
\begin{itemize}
\item  The full subtree  of $T^{\prime\prime}$ generated by the vertices $i\cdc i+m-1$ is isomorphic to $T^\prime$ and,
\item  Under the identification above, $ T^{\prime\prime}/T^{\prime}\cong T$.
\end{itemize}
\end{definition}

Let $T$ be a labeled planar tree and let $R_T$ denote the set of vertex arcs of $T$ (see $\ref{g5}$).  By definition, a rooted structure for $T$ is an element $r\in R_T$, and we denote the associated rooted tree $(T,r)$.

\begin{lemma}\label{cycmorlem}  The assignemnt $\rho_+$ defined on a labeled planar tree $T$ by
\begin{equation*}
\rho_+(T)=\sum_{r\in R_T}(T,r)
\end{equation*}
induces a morphism of operads $\rho_+\colon \cB^+\to \B^0$.
\end{lemma}

\begin{proof}  The induced map is clearly $S_n$ equivariant, so it remains to show that the operad structure is preserved under $\rho_+$.  Let $T$ and $T^\prime$ be planar trees with $n$ and $m$ vertices respectively.  Then 
\begin{equation*}
\rho_+(T\circ_iT^\prime)= \sum_{r^{\prime\prime}\in R_{T^{\prime\prime}}}\sum_{T^{\prime\prime}\in[T,T^\prime,i]}(T^{\prime\prime},r^{\prime\prime})
\end{equation*}
and
\begin{equation*}
\rho_+(T)\circ_i\rho_+(T^\prime)= \sum_{(T^{\prime\prime},r^{\prime\prime})\in[(T,r),(T^\prime,r^\prime),i]}\sum_{\substack{r\in R_T \\ r^\prime \in R_{T^\prime}}} (T^{\prime\prime}, r^{\prime\prime})
\end{equation*}
where the bracket notation $[(T,r),(T^\prime,r^\prime),i]$ with roots refers to the operadic composition of $\B^0$ as in Definition $\ref{braceopdef}$.  Let $Y$ be the set of rooted trees appearing in the first double sum and $Z$ be the set of rooted trees appearing in the second.  Since a labeled rooted tree appears at most once in either double sum, it is enough to show $Y=Z$.

Let $(T^{\prime\prime},r^{\prime\prime}) \in Y$.  The root $r^{\prime\prime}$ induces a root on any subtree, and hence on $T^\prime$ since $T^\prime$ is isomorphic to a subtree of $T^{\prime\prime}$.  The root $r^{\prime\prime}$ also induces a root after collapsing any subtree, and hence induces a root on $T$ since $T^{\prime\prime}/T^\prime \cong T$.  If these roots are called $r^\prime$ and $r$ respectively, then by construction $(T^{\prime\prime},r^{\prime\prime})\in[(T,r),(T^\prime,r^\prime),i]$, and hence $(T^{\prime\prime},r^{\prime\prime}) \in Z$.  Thus $Y\subset Z$.

On the other hand, now let $(T^{\prime\prime},r^{\prime\prime}) \in Z$.  Then there exist rooted trees $(T,r)$ and $(T^\prime, r^\prime)$ such that the rooted subtree of $(T^{\prime\prime},r^{\prime\prime})$ generated by $i\cdc i+m-1$ is isomorphic to $(T^\prime, r^\prime)$ and such that $(T^{\prime\prime}, r^{\prime\prime})/(T^\prime, r^\prime)\cong (T,r)$.  Thus as nonrooted trees the subtree of $T^{\prime\prime}$ generated by $i\cdc i+m-1$ is isomorphic to $T^\prime$ and $T^{\prime\prime}/T^\prime\cong T$, thus $T^{\prime\prime} \in [T, T^{\prime},i]$, and hence $(T^{\prime\prime},r^{\prime\prime}) \in Y$.  Thus $Z\subset Y$.
\end{proof}

The operad $\cB^+$ is concentrated in degree $0$ and acts on the (unshifted/unsuspended) coinvariants of a cyclic operad.  This is the prototype for the more interesting operad $\cB^-$ which will act on an anti-cyclic operad.

\begin{definition}  
Let $\cB^-(n)$ be the $S_n$-module spanned by labeled directed planar trees with $n$ vertices, modulo the equivalence relation generated by setting $T\sim\pm T^\prime$ if $T$ and $T^\prime$ are the same after forgetting the directed structure.  The sign is $-1$ to the number of edges whose directions disagree between $T$ and $T^\prime$.  The $S_n$ action is by relabeling (which may produce a sign).  Define the operad structure on the $\mathbb{S}$-module $\cB^-$ to be that which coincides with $\cB^+$ on the underlying unoriented graphs and which preserves the directions of the edges. 
\end{definition}

Of course, a rooted tree has a natural orientation for each of its edges; toward the root.  If $T$ is an edge oriented tree and $r$ is a choice of root, we define $\epsilon(T,r)$ to be the number of edges whose orientation does not agree with the orientation induced by $r$.

\begin{lemma}\label{cycmorlem2}  The assignemnt $\rho_-$ defined on a directed planar tree $T$ by
\begin{equation*}
\rho_-(T):=\sum_{r\in R_T}(-1)^{\epsilon(T,r)}(T,r)
\end{equation*}
induces a morphism of operads $\rho_-\colon \cB^-\to \B^0$.
\end{lemma}
\begin{proof}  In light of Lemma $\ref{cycmorlem}$ it suffices to check the signs.  Clearly the number of arrows which don't agree with the root after an insertion or permutation is the number which don't agree in the constituents before insertion or permutation.  Observe that the sign $(-1)^{\epsilon(T,r)}$ is alternating with respect to cyclically permuting the root.
\end{proof}

We can now define the cyclic brace operad.
 
\begin{definition}\label{cycbracedef}  The cyclic brace operad $\cB$ is defined to be the operad $\fr{s}^{-1}\cB^-$.  The generators of $\cB$ will be called cyclic brace operations.
\end{definition}

Observe that an immediate consequence of the definition is the existence of morphisms of operads $\fr{s}^{-1}\op{L}ie\to\cB\stackrel{\rho}\to\B$.  Also note that, as it has been defined, the cyclic brace operad $\cB$ is an operad, not a cyclic operad.  The `cyclic' in the name is motivated by Proposition $\ref{cycbrace}$.

\begin{proposition}\label{cycbrace}  Let $\op{O}$ be a cyclic operad (or non-$\Sigma$ cyclic operad).  Then $\cnsc{O}$ is naturally a $\cB$-algebra.
\end{proposition}
\begin{proof}  First, by Lemma $\ref{MSSlemma}$ it is enough to prove that if $\op{P}$ is an anti-cyclic operad then $\coprod_n \op{P}(n)_{\Z_n^+}$ is a $\cB^-$-algebra.  We define the action as follows.  Given a directed planar tree $T$ with $n$ vertices and $a_{j}\in\op{P}(m_j)$ for $j=1\cdc n$ we define $L_T=L_T(m_1\cdc m_n)$ to be the set of planar trees with tails which can be formed by adjoining tails to $T$ such that the vertex labeled by $j$ has arity $m_j$.

By Remark $\ref{circijrmk}$ any tree $l_T\in L_T$ acts on $[a_{1}]\tdt [a_{n}]$ by a composition of $\circ_{ij}$ operations.  Denote this operation by $l_T([a_{1}]\tdt [a_{n}])$.  We then define
\begin{equation*}
T([a_{1}]\tdt [a_{n}]):= \sum_{l_T\in L_T}l_T([a_{1}]\tdt [a_{n}])
\end{equation*}

This equation is the non-rooted analog of equation $\ref{braceeq}$.  
\end{proof}
Again, our extendability criterion is satisfied, hence:
\begin{corollary}\label{cycbracecor}  Let $\op{O}$ be a cyclic (or non-$\Sigma$ cyclic) operad.  Then $\cnsl{O}$ is naturally a $\cB$-algebra.
\end{corollary}

\subsubsection{Orientation data}\label{signsec}  A labeled planar tree is not a cyclic brace operation, rather a labeled planar tree determines a cyclic brace operation via a conventional choice of extra data which (in view of Section $\ref{sec:grav}$) we call orientation data.  This data may take either of two equivalent forms.  The first is a direction of each of the edges along with an order of the set of vertices, and permuting the vertices by an odd permutation or switching an edge order produces a negative sign.  This is what one extracts from Definition $\ref{cycbracedef}$.  The second is an ordering of the edges of $T$, and permuting the order by an odd permutation produces a negative sign.  That there is a natural equivalence between these two notions of orientation data can be extracted from Proposition 4.14 of \cite{GeK2}.  

To be more precise, in our context this equivalence takes the following form.  Suppose $\cB^\prime(n)$ is the graded vector space generated by planar trees with an order on the set of edges, modulo the relation that switching two edges in the order gives a negative sign, then there is a map $\rho^\prime\colon \cB^\prime\to \B$ given by summing over all roots with the sign equal to the number of disagreements between the given edge order of the nonrooted tree and the planar edge order given by the choice of root.  Taking care of signs in the suspension, one can show that the map $\rho\colon \cB\to\B$ induced from Definition $\ref{cycbracedef}$ lands in the image of $\rho^\prime$ and vice versa, and that the operad structure of $\cB$ induces an operad structure $\cB^\prime$ such that $\rho^{-1}\rho^\prime$ is an isomorphism of operads.  Taking this second view of the cyclic brace operad, we may consider it as built from trees whose edges have degree $-1$.  This gives both the correct degree and the correct $S_n$ action under this alternate description. 

Given a planar tree, we choose the following convention for specifying a cyclic brace operation.  Let $\alpha$ be the white vertex arc on the vertex labeled by $1$ which precedes (in the planar orientation) the branch containing the vertex labeled by $2$.  Choosing $\alpha$ as a root specifies both an order of the edges (the planar order starting from the root, using the rule of `first contact' as in Section $\ref{minopsec}$), as well as a direction of each of the edges (toward the root), and hence a cyclic brace operation in each of the descriptions above.  Moreover this convention is compatible with the given isomorphism between the descriptions.

\subsection{Adding a MC element.}
  
Given a cyclic operad $\op{O}$ and a MC element of the odd Lie algebra $\cnsl{O}$  (resp. $\csyl{O}$) we get a twisted differential by the general procedure: $\delta_\eta(-):=d_\op{O}(-)+\{\eta,-\}$.  In analogy with the operad case (Theorem $\ref{repthm}$), $\eta$ results in additional algebraic structure due to the fact that the MC functor is representable, as we shall now see.  Let $\textbf{Cyc}$ be the category of dg cyclic operads.

\begin{theorem}\label{cycrepthm}  Let $\op{O}$ be a cyclic operad.  There are natural bijective correspondences
\begin{equation*}
MC(\cnsl{O})\cong Hom_{\textbf{Cyc}}(\op{A}_\infty, \op{O}) \ \ \ \ \text{and} \ \ \ \  MC(\csyl{O})\cong Hom_{\textbf{Cyc}}(\op{L}_\infty, \op{O})
\end{equation*}
\end{theorem}

\begin{proof}  For example if $\rho\colon \op{L}_\infty\to \op{O}$ then by Example $\ref{cycliclex}$ we know that for $\sigma\in S_n^+$, $\sigma\ell_n:=(-1)^{|\sigma|}\ell_n$, and so we define the $S_n^+$ invariant element $\eta_n$ by:
\begin{equation*}
\eta_n:=\rho(\ell_n)\tensor \frac{(-1)^n}{n!} \in \op{O}(n)\tensor^{S_n^+}sgn_{n+1}
\end{equation*}
and define $\eta:=\prod \eta_n\in \csyl{O}$.  The fact that $\rho$ is respects the differential is then equivalent to the fact that $\eta$ satisfies the MC equation. \end{proof}

In the event that $\op{O}$ is not dg, the situation simplifies as follows.

\begin{corollary}\label{MCcorcyc}  For a linear cyclic operad $\op{O}$ (viewed as dg with zero differential and concentrated in degree $0$) there are natural bijective correspondences
\begin{equation}
MC(\cnsc{O})\cong Hom_{\textbf{Cyc}}(\op{A}s, \op{O}) \ \ \ \ \text{and} \ \ \ \  MC(\csyc{O})\cong Hom_{\textbf{Cyc}}(\op{L}ie, \op{O})
\end{equation}
\end{corollary}

Representability of the MC functors gives natural transformations between them.  These natural transformations can also be seen as being induced by diagram $\ref{cube}$.

\subsection{Cyclic cohomology and the long exact sequence}  
Let $\op{O}$ be a cyclic operad with $\eta\in MC(\cnsl{O})$.  The inclusion of dg Lie algebras
\begin{equation*}
0\to (\cnsl{O},\delta_\eta)\to (\nsl{O},\delta_\eta)
\end{equation*}
is part of a short exact sequence whose associated long exact sequence generalizes the Connes-Tsygan long exact sequence associated to a (co)cyclic module.  In particular, the pair $(\op{O},\eta)$ can be thought of as an $A_\infty$ generalization of a cyclic module, and we follow \cite{Loday}.  See Section $\ref{exsec}$ for details of particular examples.  Note that we will give the construction of the long exact sequence for $\nsl{O}$ and $\cnsl{O}$.  We could also consider the variant for the other outward pointing arrows in diagram $\ref{cube}$.  We focus on the non-symmetric case, but the symmetric case is also interesting, see \cite{GeK1}.  

To begin, observe that the odd pre-Lie identity ensures that $(\alpha\circ\eta)\circ \eta=0$ and consequently the formula
\begin{equation*}
\delta^\prime(\alpha) := \alpha\circ\eta +(-1)^{|\alpha|}\eta\circ_1\alpha 
\end{equation*}
is a differential.  The following lemma is an elementary generalization of the standard argument for cocyclic modules.
\begin{lemma}  The operators $\delta=\delta_\eta,\delta^\prime, N$ and $t$ satisfy $N\delta^\prime = \delta N$ and $(1-t)\delta = \delta^\prime (1-t)$.
\end{lemma}

This lemma allows us to define the cyclic bicomplex of the pair $(\op{O},\eta)$.

\begin{definition}  Let $\op{O}$ be a cyclic operad and let $\eta\in MC(\cnsl{O})$.  The cyclic bicomplex of $(\op{O},\eta)$, denoted $C^{\bullet,\bullet}(\op{O},\eta)$ is the bicomplex $0\to(\nsl{O},\delta)\stackrel{1-t}\to(\nsl{O},\delta^\prime)\stackrel{N}\to (\nsl{O},\delta)\stackrel{1-t}\to \dots$.  We define the cyclic cohomology of $\op{O}$ with respect to $\eta$  to be the cohomology of this bicomplex, which we shall denote by $HC^\ast(\op{O},\eta)$.
\end{definition}

\begin{proposition} 
\begin{equation*}
HC^\ast(\op{O},\eta)\cong H^\ast(\cnsl{O},\delta_\eta)
\end{equation*}
\end{proposition}
\begin{proof}  This follows as in \cite{Loday} Theorem 2.1.5.  In particular, the rows of the cyclic bicomplex have cohomology only in degree $0$, and so the cohomology of $C^{\bullet,\bullet}(\op{O},\eta)$ is the cohomology of the cokernel of the first two columns.  This is precisely the cyclic (co)invariants with the differential $\delta_\eta$. \end{proof}

We can now construct the long exact sequence.  The input for this construction is a cyclic operad which is unital; see Definition $\ref{unitaldef}$.  

\begin{proposition}\label{les}  Let $\op{O}$ be a unital cyclic operad and let $\eta\in MC(\cnsl{O})$.  There is a long exact sequence
\begin{equation*}
\dots\to H^n(\cnsl{O})\stackrel{in}\to H^n(\nsl{O})\stackrel{\Delta}\to H^{n-1}(\cnsl{O})\stackrel{S}\to H^{n+1}(\cnsl{O})\to\dots
\end{equation*}
\end{proposition}
\begin{proof}  To start, consider the short exact sequence
\begin{equation*}
0\to (\cnsl{O},\delta_\eta)\stackrel{in}\to (\nsl{O},\delta_\eta) \to coker(in)\to 0
\end{equation*}
To compute the cohomology of $coker(in)$ we consider the auxiliary short exact sequence
\begin{equation*}
0\to coker(in)\stackrel{1-t}\to(\nsl{O},\delta^\prime) \stackrel{N}\to \cnsl{O}\to 0
\end{equation*}
The fact that this sequence is exact can be seen as follows.  If $\alpha\in ker(N)$ of arity $n$, then we can define $\beta = \frac{-1}{n+1}\sum_{r=1}^nrt^r\alpha$ to see that $(1-t)\beta=\alpha$ and hence $ker(N)=im(1-t)$.  Note here we used our characteristic zero assumption.

Next observe that the complex $(\nsl{O},\delta^\prime)$ is acyclic.  Indeed, since $s_1 = -\circ_1 e$ places the degeneracy in position $1$,  it is easy to see that $\delta^\prime s_1-s_1\delta^\prime = id$.  It follows that $H^{\ast+1}(coker(in))\cong H^\ast(\cnsl{O})$, and the long exact sequence is that associated to the initial short exact sequence after this isomorphism.  This duplication of the cyclic cohomology in the long exact sequence is a manifestation of the Koszul self-duality of the associative operad.
\end{proof}
The morphism $\Delta$ in the long exact sequence will be the BV operator in Section $\ref{sec:bv}$.   One can also give a chain level description of the periodicity operator $S$, following \cite{Loday}.

\subsubsection{Cyclic cohomology of $\op{P}_\infty$-algebras.}  As seen above we can consider the cyclic cohomlogy of any cyclic or anti-cyclic operad after the choice of a MC element.  An important example occurs when we have a cyclic (resp. symplectic) $\op{P}_\infty$-algebra $A$ where $\op{P}$ is a cyclic (resp. anti-cyclic) Koszul operad.   Presently we consider this example.  This discussion is dual to the cyclic homology of $\op{P}$-algebras first defined in \cite{GeK1}.

We let $\op{P}$ be a cyclic operad or an anti cyclic operad which is also Koszul.  This implies that the Koszul dual $\op{P}^!$ is also cyclic or anti-cyclic (agreeing with $\op{P}$), and in particular $\op{P}\tensor\op{P}^!$ is cyclic.
\begin{lemma}\label{cll}  The morphism of operads $\op{L}ie\to\op{P}\tensor \op{P}^!$ (see subsection $\ref{opcosec}$) extends to a morphism of cyclic operads.
\end{lemma}
\begin{proof} 
As $S_2^+$-modules, $\op{L}ie(2)\cong sgn_3$ and $\op{P}(2)\tensor\op{P}^!(2)\cong sgn_3\tensor Hom(\op{P}(2),\op{P}(2))$, and the morphism sends $sgn_3\mapsto sgn_3\tensor id$, hence is $S_2^+$ invariant.  Call this map $\phi\colon \op{L}ie(2)\to(\op{P}\tensor \op{P}^!)(2)$, which we view as a map of $\mathbb{S}^+$-modules, taking $0$ in other arities.  There is then an inclusion of $\mathbb{S}^+$-modules $\op{P}(2)\tensor\op{P}^!(2)\to F(\op{P}(2))\tensor F(\op{P}^!(2))$, and by left adjointness of the free cyclic operad functor, a morphism of cyclic operads $F((\op{P}\tensor\op{P}^!)(2)) \stackrel{\Phi}\to F(\op{P}(2))\tensor F(\op{P}^!(2))$ which we call $\Phi$.  Thus, we have the following diagram in the category of cyclic operads:
\begin{equation*}
\xymatrix{ \op{L}ie & F(\op{L}ie(2))\ar[r]^{F(\phi) \ \ \ } \ar[l] & F((\op{P}\tensor\op{P})(2)) \ar[r]^{\Phi \ \ \ \ } & F(\op{P}(2))\tensor F(\op{P}^!(2)) \ar[r]^{ \ \ \ \ \ \ \ \pi}  & \op{P}\tensor\op{P}^!}
\end{equation*}

Now, $\phi$ extends to a morphisms of operads; in particular one can show that the Jacobi identity $J\in R_{\op{L}ie}\subset  F(\op{L}ie(2))$ is sent into $F(\op{P}(2)) \tensor R^\perp + R\tensor F(\op{P}^!(2))$ by composition in the diagram, and hence to $0$ by the projection $\pi$ (see {\it eg} \cite{LV} Lemma 7.6.6).  Since this morphism of operads is by construction $\mathbb{S}^+$ equivariant, it is a morphism of cyclic operads.
\end{proof}

\begin{remark}  A more conceptual proof of the above lemma can be had via Manin products, see \cite{GK}, \cite{ValD}, \cite{LV}.  In particular the Manin product of cyclic and anti-cyclic operads behaves as the tensor product in the sense of Lemma $\ref{detcyc}$.  The above morphisms can then be seen as the composite of the sequence $\op{L}ie\to\op{P}\circ\op{P}^!\to\op{P}\tensor\op{P}^!$, where $\circ$ represents the Manin white product.  One could also use this fact and the results of \cite{ValD} to provide another proof of the existence results for anti-cyclic structures given in \cite{Cha}. 
\end{remark}

There is a quasi-isomorphism of dg operads $D(\op{P})\stackrel{\sim}\longrightarrow \op{P}^!$.  Under the homotopy theory of cyclic operads (see \cite{KW}) it is still the case that $\op{L}_\infty$ is cofibrant and hence the morphism of cyclic operads in Lemma $\ref{cll}$ lifts to a morphism $\op{L}_\infty\to D(\op{P})\tensor\op{P}$ in the category of cyclic operads.  If we let $\op{O}=D(\op{P})\tensor\op{P}$, then Theorem $\ref{cycrepthm}$ specifies a MC element of $\csyl{O}$.  Now let $A$ be a cyclic $\op{P}_\infty$-algebra.  By definition this means that $A$ is a cyclic algebra over the cyclic (or anti-cyclic) operad $D(\op{P}^!)$.  Then we have morphisms of dg cyclic operads
\begin{equation}\label{mccoelementcyc}
\op{L}_\infty\to \op{P}^!\tensor D(\op{P}^!)\to \op{P}^!\tensor End_A
\end{equation}

This allows for an efficient definition of the cyclic cochains of a $\op{P}_\infty$-algebra.

\begin{definition}\label{ccdef}  Let $\op{P}$ be a Koszul operad and a cyclic or anti-cyclic operad and let $A$ be a cyclic $\op{P}_\infty$-algebra.  Define $\op{O}=\op{P}^!\tensor End_A$.  Then the completed cyclic $\op{P}$-cochains of $A$ is defined to be the cochain complex,
\begin{equation*}
\hat{CC}^\ast_\op{P}(A):=(\csyl{O}, \delta_\eta)
\end{equation*}
where $\delta_\eta:=d+\{\eta,-\}$ with MC element $\eta$ from Equation $\ref{mccoelementcyc}$ via Theorem $\ref{cycrepthm}$.  The cyclic $\op{P}$-cohomology of $A$, denoted $\hat{HC}^\ast_\op{P}(A)$, is defined to be the cohomology of this complex.

In the case where the MC element lifts to $\eta^\prime\in MC(\csyc{O})$ along the standard inclusion $\csyc{O}\to\csyl{O}$, we may define the noncompleted cyclic cochain complex
\begin{equation*}
CC^\ast_\op{P}(A):=(\csyc{O}, \delta_\eta)
\end{equation*}  
and denote the cohomology of this complex as $HC^\ast_\op{P}(A)$.
\end{definition}

We will be most interested in the regular case (Definition $\ref{regdef}$), where in particular the long exact sequence in Proposition $\ref{les}$ is pertinent.  Examples will be discussed below in Section $\ref{exsec}$.

\section{The BV operator and the generalized cyclic Deligne conjecture}\label{sec:bv}
The generalized Deligne conjecture (Theorem $\ref{Dconj3}$) says that given an operad $\op{O}$ and $\zeta\in MC(\nsl{O})$, the complex $(\nsl{O}, \delta_\zeta)$ is a $\op{G}_\infty$-algebra.  In this section we will show that if the operad in question is cyclic and unital then we can define a family of so-called spined brace operations which generalized Connes' boundary operator, in analogy with how the brace operations generalize Gerstenhaber's pre-Lie product.  We then show that if the Maurer-Cartan element $\zeta$ lifts along the standard inclusion $MC(\cnsl{O})\to MC(\nsl{O})$, then there is a chain model for the framed little disks which acts on $(\nsl{O}, \delta_\zeta)$ extending the $\op{G}_\infty$ structure.  In particular $(\nsl{O}, \delta_\zeta)$ is a $\op{BV}_\infty$-algebra.

\subsection{Degeneracies and Normalization}

\begin{definition}\label{unitaldef}  A unital operad is an operad $\op{O}$ along with associative compositions $\op{O}(n)\stackrel{s_i}\longrightarrow \op{O}(n-1)$ for $1\leq i\leq n$.  We call these operations degeneracies.  Here associative means we extend the usual operad associativity axioms encoded by grafting trees with tails to include erasing the tail labeled by $i$, encoded by the operation $s_i$.  The associativity can also be described by considering the degeneracies as internal: define $\op{O}(0)=\kk[e]$ and define $s_i(-) = -\circ_i e$.  A unital cyclic operad is a unital operad and a cyclic operad.  A MC element is called unital if $s_i(\zeta_n) =0$ for all $n\geq 3, i\geq 1$ and $s_i(\zeta_2)=id$ for $i=1,2$.
\end{definition}
Note no compatibility is postulated between the cyclic and unital structures.  This is because the dependency comes in the form of an extra degeneracy $s_0$, see Equation $\ref{exdeg}$.  From now on when considering MC elements associated to unital cyclic operads, we tacitly assume they are unital.

Let $\op{O}$ be a unital cyclic operad with MC element $\zeta$.  We call an element $a\in\op{O}(n)$ normalized if $s_i(a)=0$ for every $1 \leq i \leq n$.  The normalized elements form a subcomplex whose inclusion is a quasi-isomorphism and when considering chain level actions we may restrict our attention to the normalized subcomplex.

\subsection{Spined Brace Operations.}  Just as Gerstenhaber's pre-Lie product is the first in a series of higher operations (the brace operations), Connes' boundary operator $B$ may be viewed as the first in a family of higher operations, which we call spined brace operations.  The spined brace operations act on $\nsl{O}$ for any unital cyclic operad $\op{O}$.  

Recall (Appendix $\ref{sec:trees}$) that each vertex of a planar tree determines a cellular subdivision on $S^1$ whose $0$-cells correspond to the adjacent flags and whose $1$ cells are the angles of the vertex.  A spine is a distinguished cell of a white vertex, and a tree with spines is a planar rooted tree with a spine at each vertex.  We call a spine a $1$-spine or $0$-spine depending on the dimension of the distinguished cell.   The trivial spine is the outgoing $0$-cell, and a tree with spines is spineless if each spine is trivial.  Write $(T,\nu)$ for a tree with spines where $\nu$ denotes the set of spines (distinguished vertex cells) of $T$.

\begin{definition}\label{spbraceopdef}
Let $\spB^0(n)$ be the graded $S_n$-module spanned by trees with spines having $n$ vertices, and whose spines are either $1$-spines or trivial.  The degree of such a tree (before suspension) is $-$ the number of $1$-spines.  Define an operad structure on the $\mathbb{S}$-module $\spB^0$ by

\begin{equation*}
(T,\nu)\circ_i (T^\prime,\nu^\prime) := \sum_{[(T,\nu),(T^\prime,\nu^\prime),i]} (T^{\prime\prime}, \nu^{\prime\prime})
\end{equation*}
where $[(T,\nu),(T^\prime,\nu^\prime),i]$ is the set of trees with spines which can be formed by the following procedure.  Let $v$ be the vertex of $T$ labeled by $i$ and let $u$ be the vertex of $T^\prime$ adjacent to the root. If both $v$ has a spine on a $1$-cell and $u$ has a nontrivial spine the collection is empty.  Else, identify the root of $T^\prime$ with the spine of $T$, forming the new spine of $v$, and then graft the remaining $v$-branches of $T$ to $T^\prime$ such that the linear order (starting at the spine of $T$ = root of $T^\prime$) is preserved.  
\end{definition}

\begin{lemma}  As defined above, $\spB^0$ forms a graded operad. 
\end{lemma}

The graded operad $\spB^0$ may be viewed as a non-dg suboperad of $\fr{s}CC_\ast(\op{C}acti^1)$.  The salient feature of this suboperad is that it acts {\it before} adding a MC element.

\begin{definition}
Define $\spB:=\fr{s}^{-1}\spB^0$, and call $\spB$ the spined brace operad.
\end{definition}

\begin{lemma}\label{yal}  Let $\op{O}$ be a cyclic unital operad.  Then $\spB$ acts on $\nsl{O}$.
\end{lemma}
\begin{proof} This fact can be extracted from \cite{Ward}, but we will sketch the action on the normalized subcomplex.  Note the statement concerns graded operads, not dg.

Consistent with Definition $\ref{simplegluingdef}$, we will call a composition of spined brace trees simple if it occurs at a spineless vertex of maximum height.  The spined brace operad is generated under simple gluings by the (unspined) braces along with the spined brace operations of type $1$ in the parlance of \cite{Ward}.  A spined brace operation of type $1$ is a tree of height $\leq 2$ having only one vertex of height $1$, call it $v$, whose only non-trivial spine is a $1$-spine on $v$ (see Figure $\ref{fig:spinedbrace}$).  Thus, to define an action of $\spB$ on $\nsl{O}$, it suffices to define the action on these generators.  For the suboperad of unspined braces, the action is the same as was given above (Lemma $\ref{bracelem}$).  
\begin{figure}
  \begin{center}
    \includegraphics[scale=.5]{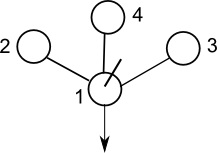}
  \end{center}	
	\caption{An example of a generating spined brace operation in $\spB(4)$.  The $1$-spine on vertex $1$ is indicated by a tic mark.}\label{fig:spinedbrace}
\end{figure} 

Since $\op{O}$ is now assumed to be both cyclic and unital, $\op{O}(n)$ also comes with an action of $t=(0\cdc n)$ and with degeneracy operators $s_i\colon\op{O}(n)\to\op{O}(n-1)$ for $1\leq i\leq n$.  In this context we define an extra degeneracy operator $s_0$ as:
\begin{equation}\label{exdeg}
s_0:= t_{n-1}s_1t_n^{-1} 
\end{equation} 

Following \cite{KCyclic}, for a generator of the spined braces the action is defined as a sum of the action of the corresponding trees with tails.  Let $T$ be such a spined brace operation of type $1$ having $n$ vertices.  Relative to a set of elements on which $T$ is to act, we adjoin tails to $T$ until achieving the proper arity at each vertex.  If $\hat{T}$ represents a decoration of $T$ with tails, then $\hat{T}(a_1\cdc a_n):= t^j s_0 T(a_1\cdc a_n)$ where $T(a_1\cdc a_n)$ is the operad action, forgetting spines, and where $j$ is the number of tails between the spine and the root in the clockwise order.

For example let $\Delta\in \spB(1)$ be the unique generator of degree $-1$.   An arrangement of tails taking the spine between flag $i-1$ and flag $i$ acts by  $t^{-i}s_0$ and so $\Delta$ in total acts by $\Delta:= N s_0$.  Analogously, it is possible to explicitly write a formula for the operation corresponding to each spined brace tree, although we don't choose to do so here, and this formula is a sum of compositions of the operations $\circ_i, s_0$, and $t$.  This formula is the unital cyclic analog of Equation $\ref{braceeq}$.
\end{proof}

\subsection{The $\op{BV}_\infty$ structure}

By the above lemma, given a unital cyclic operad $\op{O}$ with a MC element $\zeta\in\cnsl{O}$ we have a spined brace operad structure, including the operator $\Delta$ on $\nsl{O}$.  This structure accompanies the homotopy Gerstenhaber structure already constructed on $\nsl{O}$ and we can ask for an operad encoding their compatibility.  This operad was first constructed in \cite{Ward} and called $\op{TS}_\infty$ (for `trees with spines and $A_\infty$ labels') and is a chain model for the framed little disks operad, with the spined brace tree $\Delta$ inducing the BV operator.
 
\begin{theorem}\label{cycdelthm}  Let $\op{O}$ be a cyclic and unital operad and let $\eta\in MC(\cnsl{O})$.  Then there is a chain model for the framed little disks operad which acts on $\nsl{O}$ inducing a BV structure on cohomology.  In particular, in characteristic zero, $(\nsl{O},\delta_\eta)$ is a $\op{BV}_\infty$-algebra.
\end{theorem}
\begin{proof}  The operad $\op{TS}_\infty$ of \cite{Ward} is an $A_\infty$ blow-up of the Voronov's operad of Cacti and hence is a chain model for the framed little disks, $f\op{D}_2$.  By $\op{BV}_\infty$ we mean any cofibrant replacement of the operad $\op{BV}\cong f\op{D}_2$ in the model category of dg operads, eg that of \cite{GCTV}.  Hence, the formality of $f\op{D}_2$ implies the existence of a quasi-isomorphism of dg operads $\op{BV}_\infty\to \op{TS}_\infty$.  As an operad $\op{TS}_\infty$ is generated by the spined braces and the dg suboperad $\cong\TA$ having all spines trivial.  Thus, the action is defined via Theorem $\ref{actionthm1}$ and Lemma $\ref{yal}$.  To check that such an action is coherent with the relations in the endomorphism operad we appeal to the arguments in \cite{Ward}. In particular, in {\it loc.cit.} the action of $\op{TS}_\infty$ is given on the Hochschild complex of a cyclic $A_\infty$ algebra and requires only the operations $\circ_i$, $s_0$, $t$, and, $\mu_n$, along with the compatibility assured by the axioms of a unital cyclic operad with unital $A_\infty$ multiplication.  This observation and Theorem $\ref{cycrepthm}$ proves the result.

The action can be interpreted by reading a tree with a single (nontrivial) spine as a cyclic flow chart, starting at the spine and traversing clockwise.  The potential ambiguity in such a flow chart comes with a nontrivial $0$-spine, where we could either read by going up and around the given branch first or last.  The fact that $\eta$ is cyclically invariant ensures that either interpretation produces the same operation.
\end{proof}

\section{The gravity structure on the cyclic (co)invariants} \label{sec:grav}  We have seen that given a cyclic operad $\op{O}$, the cyclic brace operad acts on $\cnsl{O}$ and $\cnsc{O}$ encoding an odd Lie algebra structure.  The goal of this section is to construct a dg operad which acts on $(\cnsl{O},\delta_\eta)$ or $(\cnsc{O},\delta_\eta)$ in the presence of a Maurer-Cartan element $\eta$, encoding the cyclic brace operations  and associated evaluations by $\eta$ and then to compute its homology.  We continue to emphasize the planar case, as it has a richer structure, but a similar analysis could be applied to the non-planar case.

\subsection{The operad $\ZA$.} 
\begin{definition}  Let $T$ be an $A_\infty$ labeled tree.  Define $E_T$ to be the union of the set of white edges of $T$ and the set of black vertices of $T$.  An orientation of $T$ is a total order of $E_T$.  Two orientations are said to be equivalent if they differ by an even permutation.  Such a tree $T$ along with an equivalence class of orientations $\prec$ is said to be oriented and is denoted $T_\prec$.  For an orientation $\prec$ we let $\bar{\prec}$ denote the opposite orientation.
\end{definition}

\begin{definition}  Define $\ZA(n)$ to be the graded vector space spanned by oriented $A_\infty$-labeled trees $T$ having $n$ white vertices, modulo the relation $T_\prec = - T_{\bar{\prec}}$.
\end{definition}

The terminology `orientation' will be further justified below when we view $\ZA(n)$ as the cells of a CW complex, see Lemma $\ref{cactlem}$ and the proof of Theorem $\ref{calcthm}$.

The dg operad structure of $\ZA$ is most naturally given by comparison with $\TA$.  To begin, we construct an injection of $\mathbb{S}$-modules $\phi\colon\ZA\to\TA$.  This map extends the construction of Lemma $\ref{cycmorlem}$.  Let $T_\prec\in\ZA(n)$ and let $R$ be the set of angles of $T$, and define
\begin{equation*}
\phi(T_\prec)=\sum_{r\in R}\pm (T,r)
\end{equation*}
where the sign is the sign of the permutation sending the order $\prec$ to the planar order of the data starting at the root $r$ (in accordance with section $\ref{signsec}$).  Note here we are summing over both black and white angles in the parlance of Appendix $\ref{sec:trees}$. Consistent with section $\ref{signsec}$ we define the standard orientation of an $A_\infty$ labeled tree to be the order specified by starting at the white angle of the vertex labeled by $1$ which precedes the branch containing the vertex labeled by $2$.  In what follows, if no orientation is specified it means we assume the standard orientation.

The map $\phi$ is an injection whose image is closed under the $S_n$ action, and we give $\ZA(n)$ the inherited grading and $S_n$ action.  Moreover:

\begin{proposition}\label{dqprop}  The image $\phi(\ZA)\subset \TA$ is a suboperad.  In particular $\ZA$ inherits the structure of a dg operad from the injection $\phi$ such that $\phi\colon\ZA\to \TA$ is a morphism of dg operads.
\end{proposition}
\begin{proof}  To prove this proposition we will give a combinatorial description of the $\circ_i$ maps and the differential and check their compatibility with $\phi$.  A more topological proof can be had using the cellular structures of $\ZA(n)$ and $\TA(n)$, and we will develop this intuition later in the section.

  First, the $\circ_i$ maps are defined to extend those in the cyclic brace operations.  Namely, let $T$ and $T^\prime$ be two $A_\infty$-labeled trees having $w$ and $w^\prime$ white vertices and $b$ and $b^\prime$ black vertices.  Then we define the $\circ_i$ operation $T\circ_i T^\prime$ by the following procedure.  First label the black vertices of $T$ by ${w+1\cdc w+b}$ and label the black vertices of $T^\prime$ by ${w^\prime+1\cdc w^\prime+b^\prime}$ in any way we like, then perform a cyclic brace operation composition, in this case
\begin{equation*}
\cB(w+b)\tensor \cB(w^\prime+b^\prime)\stackrel{\circ_i}\longrightarrow \cB(w+w^\prime -1+b+b^\prime)
\end{equation*}
Finally, forget those numerical labels $\geq w+w^\prime$, all the while retaining the $A_\infty$ labels on these vertices.  If black edges are created, we contract them as usual (see $\ref{g7}$).  To check that $(\ZA, \circ_i)$, is an operad, i.e. to check associativity and equivariance of the $\circ_i$ maps, it suffices to check that the $\circ_i$ maps are compatible with $\phi$, since $\phi$ is an injection of graded $\mathbb{S}$-modules.  This follows as in the proof of Lemma $\ref{cycmorlem}$.

The fact that $\phi(\ZA)\subset \TA$ is closed under the differential of $\TA$ follows from the local nature of the differential of the minimal operad.  Namely, we can define
\begin{equation*}
\partial(T)=\sum_{v\in T} \pm\partial(T;v)
\end{equation*}
as the non-rooted analog of Equation $\ref{mindif}$ and argue that $\phi(\partial(T))=\partial(\phi(T))$.  In particular if $v$ is white,  $\partial(T;v)$ collapses white angles and if $v$ is black $\partial(T;v)$ blows up an edge is all possible ways.  We then observe that every term appearing in $\phi(\partial(T))$ also appears in $\partial(\phi(T))$.  Indeed terms in the former expression correspond to a choice of a vertex $v$, contracting white angles or blowing up an edge in $v$, and then choosing a root $r$.  Any such term can also be found by first choosing a corresponding root $r$ in $T$ and then taking the rooted differential $\partial((T,r),v)$.  

On the other hand there are terms in $\partial(\phi(T))$ which do not appear in $\phi(\partial(T))$.  This occurs if we choose a root and then apply the differential in the locality of the root to produce a rooted black vertex whose underlying unrooted black vertex is unstable.  However, such a term cancels with the term having the next choice of root (in the clockwise cyclic order) and the prior angle contracted/edge blown up. 
\end{proof}

The relationship to the brace and cyclic brace operads is the following.

\begin{corollary}\label{cormor}  The cyclic brace operad $\cB$ includes into $\ZA$ as a graded suboperad.  Moreover there are morphisms of graded operads
\begin{equation*}
\xymatrix{\cB \ar[r]^\rho \ar[d] & \B \ar[d] \\ \ZA \ar[r]^\phi & \TA }
\end{equation*}

\end{corollary}

\subsubsection{The action}  Let $\op{O}$ be a cyclic operad and $\eta\in MC(\cnsl{O})$.

\begin{theorem} \label{actthm}  $(\cnsl{O}, \delta_\eta)$ is naturally a dg $\ZA$-algebra.
\end{theorem}
\begin{proof}  By the generalized Deligne conjecture (Theorem $\ref{Dconj3}$), we know that there is a sequence of operads 
\begin{equation}
\ZA\stackrel{\phi}\longrightarrow \TA \to End_{(\nsl{O},\delta_\eta)}
\end{equation}
Viewing $\cnsl{O}\subset \nsl{O}$ we can apply the induced action of $\ZA(n)$ to get an operation in $Hom(\cnsl{O}^{\tensor n}, \nsl{O})$, and it suffices to show that this action is closed under the inclusion $\cnsl{O}\subset \nsl{O}$.  For this it is enough to consider the cyclic brace operations, since the black vertices are just evaluation on elements, and for the cyclic brace operations this follows from Corollary $\ref{cycbracecor}$.
\end{proof}

\subsection{The homology of $\ZA$ and the gravity operad.}

In this subsection we calculate the homology of $\ZA$, showing that its homology is the gravity operad of Getzler \cite{Geteq}.  This operad is equivalent (modulo degree shifts) to both the $S^1$-equivariant homology of the operad $\op{D}_2$, with $S^1$ action by rotation, or to the operad given by the homology of $\op{M}_{0,n+1}$, the moduli spaces of marked Riemann surfaces of genus zero,  (with $n+1\geq 3$).  Here we are using homological grading conventions to match the cellular intuition as in Remark $\ref{treedegremark}$.

\subsubsection{Spineless cacti}  Recall the topological $E_2$ operad of spineless cacti \cite{Vor} and variants \cite{KCacti} have been used to model string topology operations.  In particular we consider Kaufmann's cellular operad of normalized spineless cacti $CC_\ast(\op{C}act^1)$.  For each $n\geq 2$, the space $\op{C}act^1(n)$ has a free $S^1$-action by moving the base point.  Taking $S^1$ as a CW complex with a single $1$-cell, we get a map $\nabla$ by evaluating the action on said cell: 
\begin{equation*}
\nabla\colon CC_n(\op{C}act^1)\to CC_{n+1}(\op{C}act^1)
\end{equation*}
The relationship between the operads $\ZA$, $\TA$ and spineless cacti is the following.

\begin{lemma}\label{cactlem}  The image of $\nabla$ is a suboperad of $CC_\ast(\op{C}act^1)$.  This suboperad fits into a diagram of operads
\begin{equation}
\xymatrix{\ZA \ar[r]^\phi \ar[d]^{\sim} & \TA \ar[d]^{\sim} \\ im(\nabla)\ar[r] & CC_\ast(\op{C}act^1)  }
\end{equation}
where the vertical maps are weak equivalences and the horizontal maps are the canonical injections.
\end{lemma}
\begin{proof}  
To be precise, since $\nabla$ is only defined in arity $\geq 2$, we must specify that the arity $1$ term is simply the ground field $\kk$ in each operad in the diagram.  Alternatively we could work with pseudo-operads, thus allowing the arity $1$ term to be zero.  In any event the argument below is for $n\geq 2$.

Recall that the cells of $CC_\ast(\op{C}act^1(n))$ are indexed by b/w rooted trees with $n$ white vertices, and that the insertion operad structure on trees, as above, turns this correspondence into an isomorphism of operads.  After \cite{KDeligne} we call this operad of trees $\TT$.  Further recall \cite{KSch} that contracting labels of black vertices (see next paragraph) gives a weak equivalence $\TA\to \TT$.

Now to prove this lemma let us first construct an auxiliary operad $\ZZ$.  We define $\ZZ(n)$ to be the vector space spanned by (non-rooted) b/w planar trees having $n$ white vertices.  There is a map of $\mathbb{S}$-modules $\pi\colon\ZA\to\ZZ$ defined on generators as follows.  If $T\in\ZA(n)$ is a b/w tree with one or more $A_\infty$ labels of non-zero degree, then $\pi(T)=0$, and if $T$ has only $A_\infty$ labels of degree $0$ then $\pi(T)$ simply forgets the black vertex labels.  The map $\pi$ is to be thought of topologically as contracting associahedra.  In particular, the map $\pi\colon \ZA\to \ZZ$ induces the structure of a dg operad on $\ZZ$, and this map is a levelwise quasi-isomorphism.

Mimicking the construction of the injection $\ZA\to\TA$ above, we have an injection $\ZZ\to \TT$ by summing over all choices of white roots, and hence a diagram of dg operads:
\begin{equation}
\xymatrix{\ZA \ar[r] \ar[d]^{\sim} & \TA \ar[d]^{\sim} \\ \ZZ \ar[r] & \TT\cong CC_\ast(\op{C}act^1)  }
\end{equation}
and so it remains to show that the induced map $\ZZ\to CC_\ast(\op{C}act^1)$ induces an isomorphism $\ZZ\cong im(\nabla)$.  

If $T\in \ZZ(n)$ is a basis element, we enlarge the white vertices to intersect at the (possibly unstable) black vertices.  Since $n\geq 2$, there is at least one (possibly unstable) black vertex.  This graph is now a planar configuration of (topological) circles and we choose one of the intersection points, i.e. former black vertices, as a marked point.  As an unweighted planar configuration of circles with a marked point, this graph represents a cell in $CC_\ast(\op{C}act^1)$, call it $\gamma$, and it is easy to see that $\nabla(\gamma)\cong T$ under the above isomorphism.

Conversely, given a cell $\gamma\in CC_\ast(\op{C}act^1)$ there are two cases.  First if the marked point is on an intersection of lobes, we can erase the marked point, pass to the corresponding non-rooted b/w planar tree $T$ (removing the formerly marked vertex if it is unstable), and as above $\nabla(\gamma)\cong T$.  Second if the marked point is not on an intersection of lobes, then $\nabla(\gamma)=0$.  

It readily follows that this correspondence between $im(\nabla)$ and $\ZZ$ is a bijection, and thus $im(\nabla)\cong\ZZ$ under the above isomorphism.
\end{proof}

\subsubsection{The gravity operad}  We now recall the gravity operad, $\op{G}rav$.   Below we will give explicit generators for the homology of $\ZA$ as well as explicit boundaries for the relations between these generators.  As such we give a definition of $\op{G}rav$ in terms of generators and relations.

\begin{definition}\label{gpdef}  For a given $n$ we define $\upsilon_{i,j}\in S_{n}$ to be the unique permutation such that
\begin{enumerate}
\item  $\upsilon_{i,j}(1)=i, \upsilon_{i,j}(2)=j$,
\item  mod $n$, 
\begin{equation*}
\upsilon_{i,j}(3)\equiv \begin{cases} i+1 & \text{if } j\not\equiv i+1  \\ i+2 & \text{if } j\equiv i+1 \end{cases}
\end{equation*}
\item  and the cyclic order of $\{3\cdc n\}$ is preserved by $\upsilon_{i,j}$.
\end{enumerate}

\end{definition}

\begin{definition}\cite{Geteq} The gravity operad $\op{G}rav=F(E)/(R)$ is given in terms of generators and relations as follows.  For each $n\geq 2$, $E(n)$ is the trivial representation concentrated in degree $1$, call its generator $g_n$.  We consider $E(1)=0$, formally adjoining the unit if desired.  The relations $(R)$ are the operadic ideal generated by $R_{k,l}\in F(E)(k+l)$ where
\begin{equation*}
R_{k,l}= -g_{l+1}\circ_1 g_k+\sum_{1\leq i<j\leq k}\upsilon_{i,j}(g_{k+l-1}\circ_1 g_2)
\end{equation*}
for $k\geq 3, l\geq 0$, where $g_1$ is interpreted as $0$, and where $\upsilon_{i,j}\in S_{k+l}$.
\end{definition}

\begin{remark}  The insertion operad structure in the above examples, {\it eg} in $\cB$ can be extended to non connected graphs.  In this case the Gerstenhaber operad $\op{G}$ can be interpreted as the suboperad generated by the two $2$ vertex graphs of genus $0$. The operad $\op{G}rav$ is the suboperad of $\op{G}$ taking $g_k$ to be the sum of all genus $0$ one edged graphs on $k$ vertices.
\end{remark}

\begin{remark}
There is an interesting connection between the gravity operad and the $\op{L}_\infty$ operad essentially going back to \cite{CS} Theorem 6.2.  If $V$ is a gravity algebra with generating operations $g_2,g_3\cdc $ then every sublist of the generating operations can be completed by the zero operations to an $L_\infty$ algebra structure on $\Sigma V$ with zero differential.  The converse is not true; there are more relations in the gravity operad.
\end{remark}

\subsubsection{Generators of $H_\ast(\ZA)$}
\begin{remark}\label{arem}
In this section we give the graphs in $\ZA$ whose associated homology classes correspond to the gravity generators $g_k$.  Observe that if $T\in \ZA(k)$ is a tree having black vertex labels of degree $0$ which happens to be a cycle, then the class $[T]$ does not change if we change the black vertex labels of $T$ to other labels of degree $0$.  As such, we repress the black vertex labels below.  Alternatively, we could consider these trees as having no black vertex labels and thus as generators for the homology of $\ZZ(k)$.  The two approaches are equivalent under the above quasi-isomorphism $\ZA(k)\stackrel{\sim}\to \ZZ(k)$.
\end{remark}
\begin{definition}
For $k\geq 3$, let $\gamma_k\in \ZA(k)$ be the b/w tree having one black vertex of degree $0$, $k$ edges, and $k$ white vertices such that each white vertex is connected to the black vertex by an edge and such that the white vertices are labeled $\{1\cdc k\}$ clockwise in the planar order.  See Figure $\ref{gammafig}$.  In the case $k=2$ the above description produces an unstable black vertex which is erased to define $\gamma_2$.  In particular $\gamma_2\in \cB(2)$ is the Lie bracket in the cyclic brace operad. \end{definition}
	
	\begin{figure}	\centering
		\includegraphics[scale=.2]{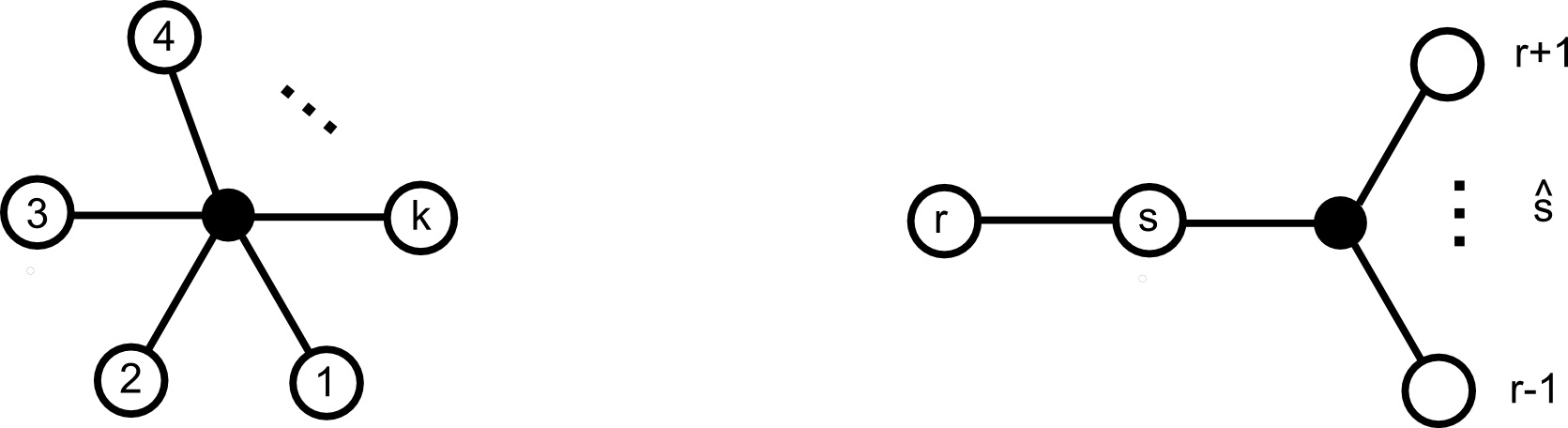}
				\caption{Left: $\gamma_k$.  Right:  $\zeta_k(r,s)$.} \label{gammafig}
	\end{figure}

\begin{lemma} \label{anotherlemma} The graph $\gamma_k\in \ZA(k)$ is a cycle of degree $1$.  Furthermore for any $\sigma\in S_k, [\gamma]= [\sigma\gamma]$.
\end{lemma}
\begin{proof}  Under homological grading conventions (see Remark $\ref{treedegremark}$), the degree of $\gamma_k$ is $1$ by definition, and since $\ZA(k)$ is concentrated in degree $\geq 1$ it is a cycle.  For the latter statement we use the graph $\zeta_k(r;s)$, pictured in Figure $\ref{gammafig}$ and defined formally in Definition $\ref{tcpdef}$. In particular $\partial(\zeta_k(r,s))$ is a boundary interchanging the order of $r$ and $s$.
\end{proof}

We have not yet argued that the $\gamma_k$ constitute all generators.  This will be shown below.

\subsubsection{Explicit relations in $\ZA$ and $\ZZ$}  We will now see that the cycles $\gamma_k\in \ZA(k)$ satisfy the relations of the gravity operad up to explicit homotopies.  The boundary giving the relation $R_{k,l}$ on homology will be denoted $\rho_{k,l}$ and is given as a signed sum of graphs.  For simplicity we will give the explicit homotopies in $\ZZ$, which can be augmented to explicit homotopies in $\ZA$ by including trees with the appropriate degree $1$ (edge) labels to move between the different degree $0$ (vertex) labels which arise in the compositions in the $\op{A}_\infty$ case, see Remark $\ref{arem}$.

Since the definition of the $\rho_{k,l}$ is somewhat involved, let me summarize it before we begin.  The indexing set of the sum $\rho_{k,l}$ is combinatorial in nature, and is introduced in two steps.  First we introduce the notion of a `wheel' and an operator $tr$ whose orbit indexes $\rho_{k,l}$ up to simultaneous relabeling of (morally) half of the terms (Definition $\ref{tcpdef}$).  This relabeling can not be done term by term, but it can be done simultaneously, and for this we define `interchangers' $\iota_{k}$ (Definition $\ref{intdef}$).  We then define $\rho_{k,l}$ as the orbit of $tr$ along with the $\iota_{k}$ graphs, and show that $\partial(\rho_{k,l})$ is precisely the corresponding gravity relation in Lemma $\ref{hardlem}$.

\begin{definition}\label{tcpdef}

\begin{enumerate}

\item  If $A$ and $B$ are disjoint finite ordered sets we let $A\sqcup B$ denote the ordered set whose elements are $A\cup B$ and whose order is colloquially `$A$ then $B$'.

\item  Let $1\leq m \leq n$ and $1\leq d\leq n-1$.  A wheel of arity $n$ of degree $d$ with center $m$ is a partition of the set $\{1\cdc n \}\setminus\{m\}$ into $d$ nonempty ordered subsets $A_1\cdc A_d$ such that the cyclic order of $A_1\sqcup\dots\sqcup A_d$ agrees with the natural cyclic order on $\{1\cdc n \}\setminus\{m\}$ and such that $m+1\in A_1$.  The set of such wheels is denoted $\op{W}_n^d(m)$.  (Note the total order of $A_i$ need not agree with that induced from the integral order.)  

\item  The transfer operator $tr$ is an assignment
\begin{equation}
tr\colon\op{W}_n^d(m)\to \op{W}_n^d(m) 
\end{equation}
defined provided $\text{min}(A_1)\neq m+1$ and given by $tr(A_1)=A_1\setminus \text{min}(A_1)$, $tr(A_2)=A_2\sqcup \text{min}(A_{3})$, and $tr(A_i)= (A_i\setminus \text{min}(A_i))\sqcup \text{min}(A_{i+1})$ (mod $d$) for $3\leq i\leq d$.
\item  For $n\geq 4$ we define $\omega_n(m)\in \op{W}^3_n(m)$ to be the wheel with $A_1=\{m+4\cdc m-1,m+1\}, A_2=\{m+2\}, A_3=\{m+3\}$ (mod $n$).
\item  For $n\geq 3$ and $1\leq r,s\leq n$ we define $\zeta_n(r;s)\in \op{W}^2_n(s)$ to be the unique wheel with $A_i={r}$ for some $i$.
\item  For a fixed $k\geq 3$ we define $\epsilon_{r,s}\in\{0,1\}$ so that $\upsilon_{r,s}\zeta_k(1,2)= (-1)^{\epsilon_{r,s}}\zeta_k(r,s)$.
\end{enumerate}

\end{definition}

\begin{figure}
	\centering
		\includegraphics[scale=.35]{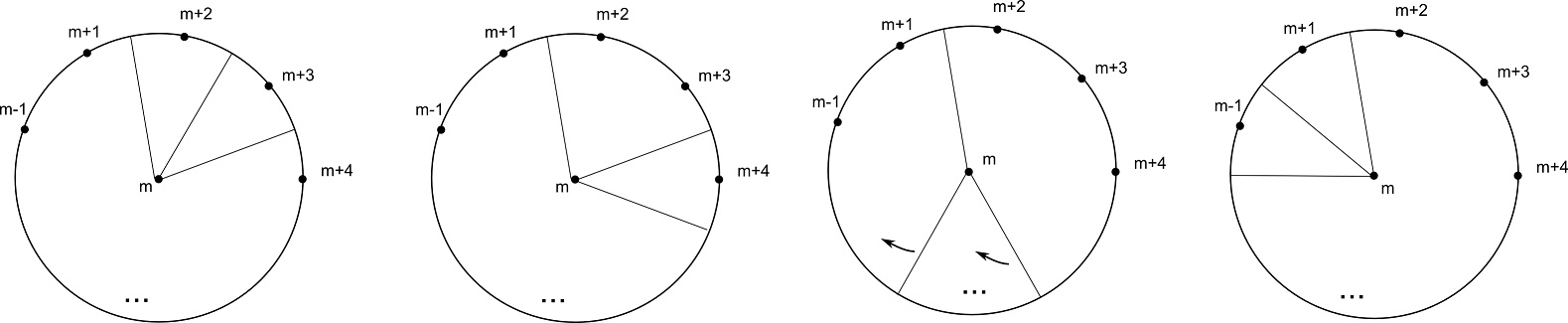}
	\caption{The orbit of $tr$ applied to $\omega_n(m)$.  From left: the wheel $\omega_n(m)$; $tr(\omega_n(m))$; $tr^j(\omega_n(m))$; $tr^{n-4}(\omega_n(m))$}
	\label{fig:morewheels}
\end{figure}

We represent wheels pictorially as follows.  Draw a circle with $n-1$ marked points labeled clockwise in the cyclic order by $\{1\cdc n\}\setminus\{m\}$ and label the center of the circle by $m$.  Then draw $d$ line segments connecting $m$ to the circle and not intersecting the marked points, such that the labeled marked points on the circle between two line segments are exactly the sets $A_i$ in order.  Clearly every wheel can be uniquely represented by such a diagram.

\begin{figure}
	\centering
		\includegraphics[scale=.2]{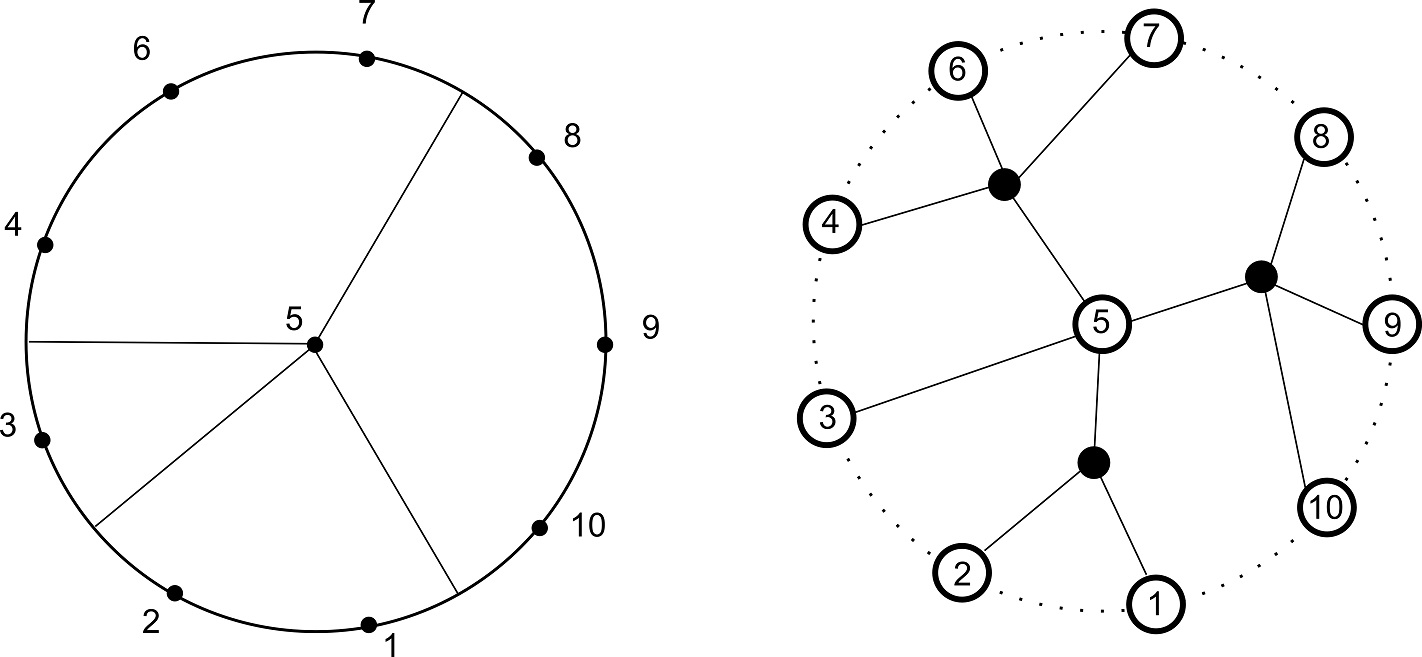}
	\caption{Left: a wheel in $\op{W}_{10}^4(5)$.  Right: its dual graph.}
	\label{fig:awheel}
\end{figure}

To a wheel $\op{W}^d_n(m)$ we associate a degree $d$ element of $\ZZ(n)$ by the following dual graph construction.  Place a white vertex over each of the $n$ marked points, including the center.  Place a black vertex in each of the $d$ sectors of the circle, as specified by the line segments.  Connect each black vertex to the center and to any white vertices that appear in the boundary in its sector (there is always at least one).  Finally, erase any unstable black vertices.  By abuse of notation we consider $\kk[\op{W}^d_n(m)]\subset\ZZ^d(n)$. (See Figure $\ref{fig:awheel}$).  Such wheel complexes form a resolution of the generator $\gamma_n$, in the sense of the following lemma.

\begin{lemma}\label{wlem}  Under the above identification the differential $\partial$ of $\ZZ^d(n)$ sends $\partial\colon \kk[\op{W}^d_n(m)]\to \kk[\op{W}^{d-1}_n(m)]$.  In particular a wheel is sent to a signed sum of wheels given by removing one line segment in each of the $d$ possible ways.  The homology of the complex $\kk[\op{W}^\ast_n(m)]$ is concentrated in degree $1$ and is generated by $\gamma_n$.
\end{lemma}

\begin{definition}\label{intdef}  An interchanger of arity $n$ is a subset of $\{1\cdc n\}$ of size $4$, which we depict as a circle with marked $n^{th}$ roots of unity having two intersecting line segments terminating at the $4$ chosen points.  
\end{definition}

To an interchanger $\iota=i<r<j<s$ of arity $n$ we associate a graph by the following procedure.  Define the following partition of $\{1\cdc n\}\setminus \{i,r,j,s\}$,  $A =\{m: i<m<r\}, B =\{m: r<m<j\}, C =\{m: j<m<s\}, D =\{m: s<m\}, E =\{m: m<i\}$.  Now to form the graph, first attach white vertices labeled by $i$ and $r$ to a central black vertex.  Second, attach white vertices $j$ to $i$ and and $s$ to $r$.  Finally, attach the remaining white vertices to the central black vertex such that the cyclic order of the white vertices attached to the central black vertex agrees with $\{i \}\sqcup A\sqcup C\sqcup\{r\}\sqcup D \sqcup B \sqcup E$.  See Figure $\ref{fig:int}$ for an example.  In the case that $n=4$ the black vertex is unstable and is not drawn.  Define $\op{I}(n)$ to be the set of all such graphs.  We will view $\kk[\op{I}(n)]\subset \ZZ(n)$ but interchangers do not take the standard orientation, rather their orientation is fixed in the proof of Lemma $\ref{intlem}$ below.

\begin{figure}
	\centering
		\includegraphics[scale=.15]{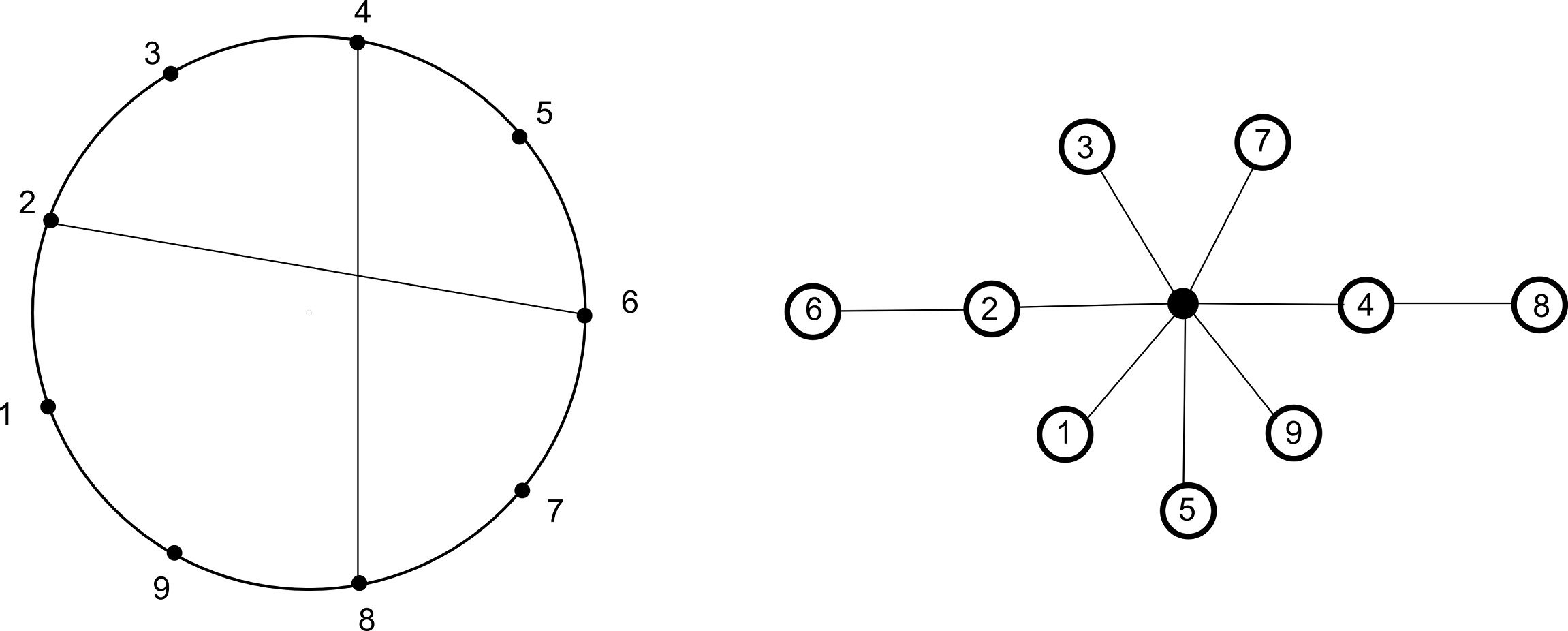}
	\caption{An interchanger in $\op{I}(9)$ and its associated graph in $\ZA(9)$.}
	\label{fig:int}
\end{figure}

\begin{lemma}\label{intlem}  Write $\tau_{r,s}$ for the transposition $(r,s)$.  Then
\begin{equation}\label{inteq}
\partial(\sum_{\iota\in\op{I}(n)}\iota)= \sum_{1\leq r<s\leq n}(-1)^{\epsilon_{s,r}}\zeta_n(s;r) - (-1)^{\epsilon_{r,s}}\tau_{r;s}\zeta_n(r;s)
\end{equation}
\end{lemma}
\begin{proof}  This is a straightforward but lengthy calculation.  First, if $r$ and $s$ are cyclically adjacent then $\upsilon_{s,r}= \tau_{r,s}\upsilon_{r,s}$ and the corresponding terms cancel.  Now fix $r<s$ which are not cyclically adjacent and consider all interchangers with the chord $r,s$.  For each interchanger, $\partial(\iota)$ has four terms, corresponding to the four collapsible white angles.  Each term in the differential which does not contract an angle on $r$ appears twice with opposite sign, with the exception of one unique contribution each from those configurations that have the opposite chord as short as possible on each end.  These two produce two noncancelling terms which are those above for the given $r,s$.  

Regarding signs and orientation, we view an interchanger celluarly as a square with weights on the white angles which sum to $1$ at each vertex.  (See the proof of Theorem $\ref{calcthm}$ for more detail of the cellular structure of $\Sigma^{-1}\ZZ$.)  The sum $\Sigma \iota$ fits together to form a highly symmetric cubical subdivision of an $n$-gon, and we orient this configuration to induce an orientation of each interchanger.  In particular each boundary edge is subdivided into $n-3$ $1$-cells and the boundary of the $n$-gon corresponds to the $n(n-3)$ noncancelling terms in line $\ref{inteq}$.  The example $n=8$ is pictured.
\begin{equation*}
		\includegraphics[scale=.3]{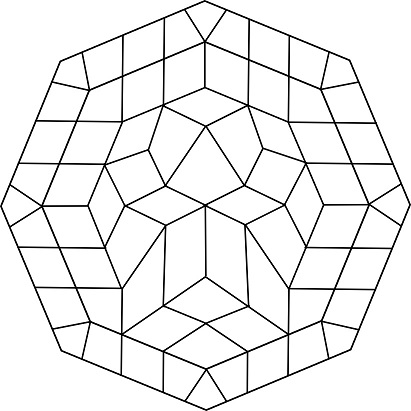}
\end{equation*}

The signs in equation $\ref{inteq}$ are dictated (up to an overall sign) by the fact that the $S_n$ action is compatible with the boundary operator along with the fact that $\partial^2=0$.  In particular we observe that $\partial((-1)^{\epsilon_{s,r}}\zeta(s;r))=\upsilon_{s,r}\partial(\zeta(1;2))$ and that $S_n$ acts without signs on the terms in $\partial(\zeta(1;2))$ (Lemma $\ref{anotherlemma}$).  The overall sign is fixed by the correct choice of orientation in the preceding paragraph.
\end{proof}

\begin{definition}  Let $T$ be a planar tree in $\ZZ(n)$ whose vertex labeled by $n$ is a leaf.  We define $T^{\vee l}\in \ZZ(n+l-1)$ to be the tree given by turning said vertex black and grafting on $l$ adjacent white vertices labeled in the planar order from $n$ to $n+l-1$ to said vertex.
\end{definition}

\begin{definition}  For $k+l\geq 4$ we define $\rho_{k,l}\in \ZZ(k+l)$ by
\begin{equation}
\rho_{k,0}=\sum_{m=1}^k\sum_{j=0}^{k-4} (-1)^j tr^j(\omega_k(m))-\sum_{\iota\in\op{I}(k)}\iota
\end{equation}
if $l=0$ and
\begin{equation}
\rho_{k,l}=\sum_{m=1}^k\sum_{j=0}^{k-3} (-1)^j (tr^j(\omega_{k+1}(m)))^{\vee l}-\sum_{\iota\in\op{I}(k+1)}\iota^{\vee l}
\end{equation}
if $l\neq 0$.  We also define $\rho_{3,0}=0$.
\end{definition}

\begin{lemma}\label{hardlem}  The elements $\rho_{k,l}$ are explicit homotopies measuring the failure of the gravity relations on $\ZZ$.  Precisely this means:
\begin{equation}\label{exeq}
\partial(\rho_{k,l}) =-\gamma_{l+1}\circ_1 \gamma_k+\sum_{1\leq i<j\leq k}\upsilon_{i,j}(\gamma_{k+l-1}\circ_1 \gamma_2)
\end{equation}
\end{lemma}
\begin{proof}
Note that if $k=3$ and $l=0$ the fact that $\rho_{3,0}=0$ is equivalent to $\gamma_2$ satisfying the Jacobi identity strictly in $\ZZ$, i.e. on the chain level.  Of course the graph $\gamma_2\in \cB(2)\subset \ZZ(2)$ is exactly the Lie bracket in $\cB$ and the Jacobi identity has been verified above.  In particular Equation $\ref{exeq}$ holds in this case.

Let us first consider the case $k\geq 4$ and $l=0$.  For a given $m$, by Lemma $\ref{wlem}$ we know that $\partial(tr^j(\omega_k(m)))$ consists of three terms given by erasing one of the three line segments.  In $\sum_{j=1}^{k-3} \pm tr^j(\omega_k(m))$, the line segment landing between $m+1$ and $m+2$ is fixed, and therefore all of the terms in $\partial(\sum_{j=1}^{k-3} \pm tr^j(\omega_k(m)))$ which do not remove this line segment appear twice with opposite sign, unless $j=0$ or $j=k-4$.  The only remaining terms are those with a sector containing exactly one marked point.  In particular,
\begin{equation*}
\partial(\sum_{j=0}^{k-4} (-1)^j tr^j(\omega_k(m))) = \sum_{\substack{1\leq r\leq k \\ r\neq m}}(-1)^{\epsilon_{r,m}} \zeta_k(r;m)
\end{equation*}
Here, once again, we view $\sum_{j=0}^{k-4} (-1)^j tr^j(\omega_k(m))$ cellularly as an oriented contractible complex given by gluing triangles along common faces.  The sign $(-1)^{\epsilon_{r,m}}$ is then dictated by the fact that $\partial^2=0$ along with the fact that $\partial((-1)^{\epsilon_{r,m}}\zeta(r;m))=\upsilon_{s,r}\partial(\zeta(1;2))$ and the fact that $S_n$ acts without signs on the terms in $\partial(\zeta(1;2))$ (Lemma $\ref{anotherlemma}$).

Those terms in the above expression with $r<m$ are desired, and those terms with $r>m$ are accommodated by the interchangers as follows.  We first compute
\begin{equation*}
\gamma_{k-1}\circ_1 \gamma_2= \zeta_k(1;2)-\zeta_k(2;1)= \zeta_{k}(1;2)+\tau_{1,2}\zeta_{k}(1;2) \in \kk[\op{W}^2_k(1)]\oplus \kk[\op{W}^2_k(2)]
\end{equation*} 
where $\tau$ continues to mean transposition.  Applying $\upsilon_{r,s}$ we see that
\begin{equation}\label{permeq}
\upsilon_{r,s}(\gamma_{k-1}\circ_1 \gamma_2) = (-1)^{\epsilon_{r,s}}( \zeta_{k}(r;s)+ \tau_{r,s}\zeta_{k}(r;s))
\end{equation}
Thus, $\partial(\sum_{m=1}^k\sum_{j=0}^{k-4} (-1)^j tr^j(\omega_k(m))) =$
\begin{align*}
& \sum_{\substack{1\leq r<s \leq k}}(  (-1)^{\epsilon_{r,s}}\zeta_k(r;s) +   (-1)^{\epsilon_{s,r}}\zeta_k(s;r)) \\
& \ \ =  \ds\sum_{1\leq r<s \leq k} (-1)^{\epsilon_{r,s}}(\zeta_k(r;s)+ \tau_{r,s}\zeta_k(r;s)) +  \ds\sum_{1\leq r<s \leq k}  ( (-1)^{\epsilon_{s,r}}\zeta_k(s;r)- (-1)^{\epsilon_{r,s}}\tau_{r,s}\zeta_k(r;s)) \\
& \ \ \ \ \ \ \ \ \  = \ds\sum_{1\leq r<s\leq k}\upsilon_{r,s}(\gamma_{k-1}\circ_1 \gamma_2) +\sum_{\iota\in \op{I}(k)} \partial(\iota)  
\end{align*}
from Lemma $\ref{intlem}$ and Equation $\ref{permeq}$.  Hence Equation $\ref{exeq}$ is satisfied.

The calculation works similarly when $l\geq 1$, once we observe that the $(-)^{\vee l}$ operation commutes with the differential.  The extra terms (in the above indexing notation) given by $s=k+1$ correspond to the terms in the boundary of $\gamma_{l+1}\circ_1 \gamma_k$. \end{proof}

\subsubsection{Completing the calculation.}

\begin{theorem}\label{calcthm} The assignment $g_k\mapsto [\gamma_k]$ induces an isomorphism of operads $\psi\colon\op{G}rav\stackrel{\cong}\to H_\ast(\ZA)$.
\end{theorem}
\begin{proof}  The fact that the above assignment is a morphism of operads follows from Lemma $\ref{hardlem}$.  Thus it suffices to show that this assignment induces a levelwise isomorphism.  As such we fix an arbitrary arity $n\geq 2$ for the remainder of the argument.  Let us first show that $\psi$ is a levelwise injection.  For this it is enough to show that the following diagram commutes,
\begin{equation} 
\xymatrix{\op{G}rav(n) \ar[d]^{\psi} \ar@{>>->}[drr]^{in} && \\ H_\ast(\ZA(n)) \ar[rr]_{H_\ast(\phi) \ } && H_\ast(\op{D}_2(n))\cong \op{G}(n)}
\end{equation}
where the diagonal arrow is the canonical injection.  Explicitly this means $[\phi(\gamma_n)]=\sum_{i<j}\{a_i,a_j\}\cdot a_1\cdot ...\cdot a_n$.  This can be seen by induction as follows.  If $n=2$ the statement is clearly true.  Provided that the statement holds up to $n-1$, we conclude that in $\op{G}$ we have $[\phi(\gamma_{n-1})]\circ_{n-1}[\mu_2]=in(g_{n-1})\circ_{n-1}[\mu_2]$, where $\mu_2\in\TT(2)$ is the corolla representing the associative product.  Thus, it is sufficient to construct an explicit homotopy between $\phi(\gamma_n)-\phi(\gamma_{n-1})\circ_{n-1}\mu_2$ and the canonical representative of $g_n-(g_{n-1}\circ_{n-1}[\mu_2])$.  This homotopy may be constructed explicitly by assembling the following $2$-cells:
\begin{equation*}\label{fig:homotopy}
		\includegraphics[scale=.2]{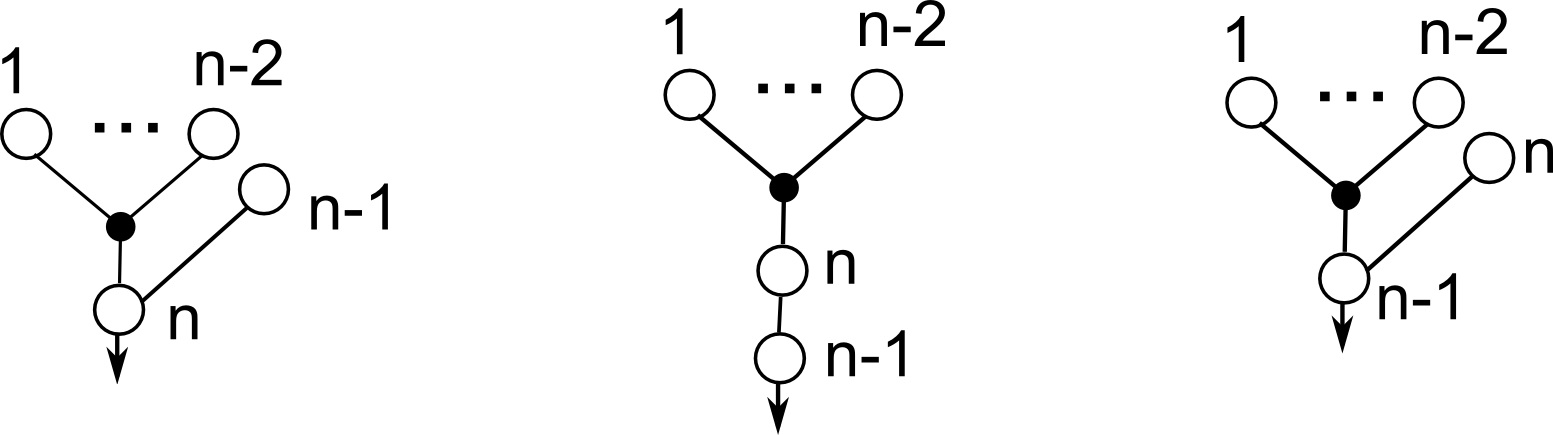}
\end{equation*}

Thus $g_k\mapsto [\gamma_k]$ induces a levelwise injection and to conclude that it is an isomorphism, it is enough to argue $\op{G}rav(n)$ and $H_\ast(\ZA(n))$ have the same rank.  To this end we observe that $\Sigma^{-1}\ZZ(n)$ can be viewed as the cellular chains of a CW model for $\op{C}act^1(n)/S^1$.  To see this we abbreviate $X=\op{C}act^1(n)/S^1$ and define the $m$-skeleton $X^m$ by declaring a point $[q]\in X$ to be in $X^m$ iff there is a representative $p$ of $[q]$ with $\leq i+n$ total arcs in the cactus.  Then, given a basis element in $\Sigma^{-1}\ZZ(n)$ of degree $m$, we either choose a black vertex to which we affix a root or, if there are no black vertices, add an unstable black vertex to a white edge and attach the root to it.  This determines an $m$ cell in $\op{C}act^1(n)$, and hence a map
\begin{equation*}
D_m\to\op{C}act^1(n)\to \op{C}act^1(n)/S^1
\end{equation*}
where $D_m$ is an $m$-disk.  The composite map is independent of the choice of root and lands in $X^m$.  The fact that these cells form a CW structure for $\op{C}act^1(n)/S^1$ follows the CW structure of $\op{C}act^1(n)$ and the fact that the set of boundary components (white angles) is not altered by the choice of a black root.  Thus $H_\ast(\ZA(n))\cong \Sigma H_\ast(\Sigma^{-1} \ZA(n))\cong \Sigma H_\ast(\op{C}act(n)/S^1)\cong \op{G}rav(n)$.  We conclude that the levelwise injection $\psi$ is a levelwise isomorphism, and hence an isomorphism of operads. \end{proof}

\begin{corollary}\label{equiDel}($S^1$-equivariant Deligne Conjecture)  Let $\op{O}$ be a cyclic operad and let $\eta\in MC(\cnsl{O})$.  Then $\ZA$ is a chain model for the $S^1$-equivariant homology of $\op{D}_2$ which acts on $(\cnsl{O},\delta_\eta)$ inducing a gravity algebra structure on cohomology.
\end{corollary}

The use of the term `chain model' in the previous corollary is justified by the fact that the equivalence $\op{C}act\stackrel{\sim}\longleftarrow \dots \stackrel{\sim}\longrightarrow \op{D}_2$ induces an isomorphism of operads after applying $\Sigma H^{S^1}_\ast(-)$ (see \cite{West} Lemma 7.8 and Theorem 7.9), along with the fact that $\ZA\stackrel{\sim}\rightarrow \ZZ \cong \Sigma CC_\ast(\op{C}act^1(n)/S^1)$.

The fact that the chains $\Sigma^{-1}\ZZ$ on $\op{C}act^1(n)/S^1$ do not form a dg operad, but rather an odd operad (Definition $\ref{odddef}$), was expected in light of \cite{KWZ}.  Topologically, there are two ways to view $\ZZ$.  First we can interpret generating trees as specifying the collection of cells in $CC_\ast(\op{C}act^1)$ that are formed by adjoining all possible white roots.  These cells are glued together along the adjacent black roots.  See Figures $\ref{fig:cylinder}$ and $\ref{fig:both}$.  Alternatively we can view the generating trees as specifying cells in the quotient $\op{C}act^1(n)/S^1$, which gives the desuspended degree.  The former is operadic but not cellular.  The latter is cellular but not operadic.  In moving from $\ZZ$ to $\ZA$, the cellular interpretation is the usual topological blow-up.

\begin{figure}	
	\centering
		\includegraphics[scale=.33]{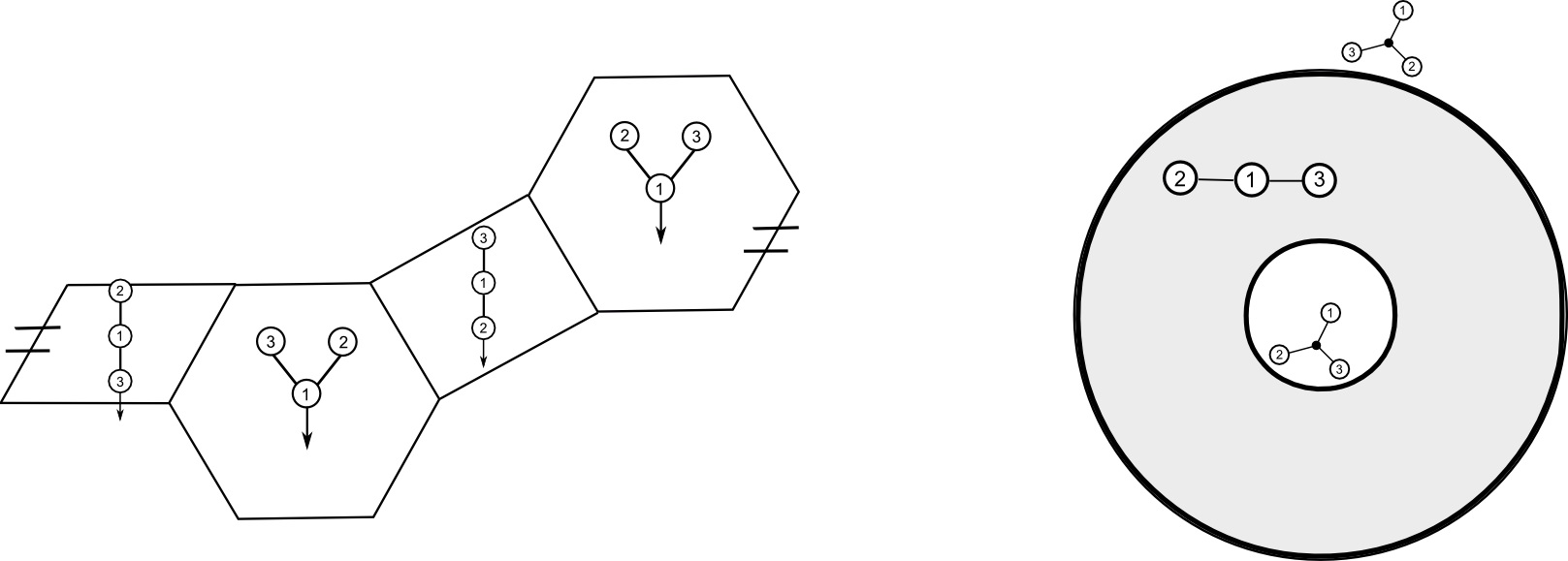}
		
	\caption{Gluing together cells in $\TA$ to form two cycles in $\ZA$ and the cyclic brace operation between them.  Notice the two boundary components of the cylinder on the left hand side have $6$ terms each which correspond precisely to the choices of root in the boundary components on the left hand side.}
\label{fig:cylinder}
\end{figure}

\begin{figure}
	\begin{center}
		\includegraphics[scale=.33]{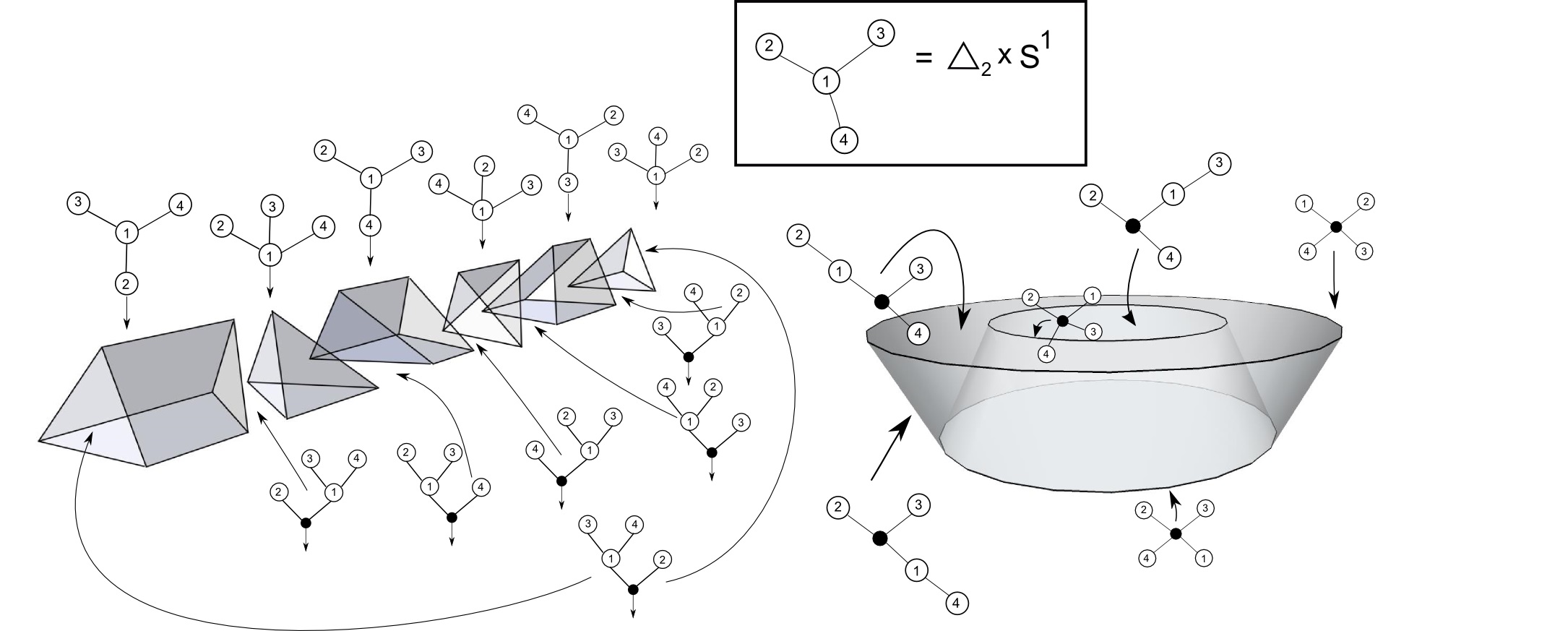}
	\caption{The cyclic brace operation pictured in the box maps to $\TT\cong CC_\ast(\op{C}act^1)$ as a sum of $6$ cells, $3$ tetrahedra and $3$ triangular prisms.  These cells are glued together along faces having black roots.  The result is $\Delta_2\times S^1$.  In $\ZA$ the blow up results in $W_3\times S^1$, where $W_3$ is the $2$-dimensional cyclohedron (a hexagon).  On the right hand side, the $3$ homologous $1$-cycles represent the homology class corresponding to $g_4\in\op{G}rav(4)$.}
	\label{fig:both}
	\end{center}
\end{figure}

\subsection{Connection to Tw}  To conclude this section we will compare the above constructions with the twisting construction Tw of T. Willwacher in \cite{Will1}.  In particular, the existence of a morphism $\fr{s}\op{L}ie\to \cB$ qualifies the cyclic brace operad as an input for Tw and we will show that the operad $\ZA$ is quasi-isomorphic to $\text{Tw}(\cB)$.  The argument here-in is a nearly verbatim adaptation of those in Dolgushev and Willwacher \cite{VW}.

For an explanation of Tw we refer the reader to \cite{Will1} and \cite{DR}, as we will only recall the most relevant of particulars here.  By definition $\text{Tw}(\cB)$ is a dg operad whose arity $n$ component can be written
\begin{equation}\label{tweq}
\text{Tw}(\cB)(n) =\ds\bigoplus_{r=0}^\infty \Sigma^{r}\cB(n+r)^{S_r}
\end{equation}
where $S_r$ acts on the labels $n+1\cdc n+r$.  We can therefore consider $\text{Tw}(\cB)(n)$ to be the vector space spanned by all planar trees with $n+r$ vertices, partitioned into two sets $V_\text{white}$ of size $n$ and $V_\text{black}$ of size $r$ where the white vertices are labeled by $\{1\cdc n\}$ and the black vertices are not labeled.  Note that such trees are not in general `b/w' in the parlance of Appendix $\ref{sec:trees}$; they can have black edges and unstable black vertices, and we say such trees are of `Tw-type'.  The operad structure is then the insertion operad structure inherited from $\cB$.    Equation $\ref{tweq}$ takes the total grading so that the degree of such a tree is
\begin{equation*}
|T|=2|V_\text{black}|-|E(T)|
\end{equation*}

The differential has the following combinatorial description.  Let $T_{\bullet\circ}\in \text{Tw}(\cB)(1+1)$ be the unique tree with one white and one black vertex, we let $T_{\bullet\bullet}\in \text{Tw}(\cB)(0+2)$ be the unique tree with two black vertices, and given a tree $T\in  \text{Tw}(\cB)(n+r)$ with a black vertex $v$ we let $T_v\in  \text{Tw}(\cB)(n+1+r-1)$ be the tree formed from $T$ by switching the color of $v$ and labeling the result by $n+1$.  Then we define
\begin{equation*}
\partial^{\text{Tw}}(T)=-(-1)^{|T|}\ds\sum_{v\in V_\text{black}} T_v\circ_{n+1} T_{\bullet\bullet} + T_{\bullet\circ}\circ_1 T -(-1)^{|T|}\ds\sum_{i=1}^n T\circ_i  T_{\bullet\circ}
\end{equation*}
This is the non-rooted analogue of equation 8.14 of \cite{VW}.  Note that $T_v\circ_{n+1} T_{\bullet\bullet}$ has the effect of blowing up a black edge at the given black vertex and $T\circ_i  T_{\bullet\circ}$ has the effect of blowing up a mixed edge at the white vertex labeled by $i$.

Let us now define a dg map $\ZA(n)\hookrightarrow \text{Tw}(\cB)(n)$ in the obvious way.  Namely an $A_\infty$ labeled tree $T$ can be interpreted as a Tw-type tree if we blow up the black vertex according to its label.  By abuse, we refer to this morphism as the canonical inclusion.  The image of this map is simply those Tw-type trees which have no black vertices of valence $<3$.  We call this property (after \cite{VW}) admissibility.  This property is closed under the operad structure, and so this assignment is a morphism of operads.  To see that it is compatible with the differential we may observe (after \cite{VW} p. 73) that if the differential $\partial^{Tw}(T)$ is applied to an admissible tree, while valence one and two black vertices can be created, they appear in canceling pairs.  Indeed a valence two black vertex appears on an edge with a contribution from each of the adjacent vertices, while a valence $1$ black vertex appears with a contribution from the single adjacent vertex and a contribution from $T_{\bullet\circ}\circ_1 T$.

Recall that the notion of a reduced operad considers only arities $n\geq 1$.  To compare $Tw(\cB)$ with $\ZA$ we will from now on forget the arity $0$ term and consider $Tw(\cB)$ to be a reduced operad.  Of course, the graph complex $Tw(\cB)(0)$ is an interesting algebra in its own right, but it is not the object of our study here.

\begin{theorem}  The canonical inclusion $\ZA\hookrightarrow \text{Tw}(\cB)$ is a weak equivalence of reduced operads.

\end{theorem}
\begin{proof}  Being a morphism of dg operads, it remains to show that the canonical inclusion is a levelwise quasi-isomorphism.  Fix an arity $n\geq 1$.  We continue to follow \cite{VW}.   Let $\nu_{\leq 2}(T)$ be the number of black vertices of valence 1 or 2 for a Tw-type tree $T$.  Define a filtration
\begin{equation*}
\dots \subset \op{F}^{m-1}\text{Tw}(\cB)\subset \op{F}^{m}\text{Tw}(\cB)\subset \dots
\end{equation*}
where $\op{F}^{m}\text{Tw}(\cB)$ is the span of those trees having $\nu_{\leq 2}(T)-|T|\leq m$.  Note that since the differential increases $\nu_{\leq 2}(T)$ by at most $1$, the filtration is closed under $\partial$.  Turning then to the associated graded complex we see terms in the differential that introduce a univalent black vertex appear in pairs, and so the only terms in the differential which survive the quotienting are those which introduce a bivalent black vertex.  As such we may write,
\begin{equation*}
Gr_\op{F}(\text{Tw}(\cB)) = \bigoplus_{\text{cores } \Gamma} Tw(\cB)_\Gamma
\end{equation*} 
where a core is a Tw-type tree with no bivalent vertices and where $\text{Tw}(\cB)_\Gamma$ is the span of those trees which can be formed by adding zero or more bivalent vertices to $\Gamma$.  Now we can write 
\begin{equation*}
Tw(\cB)_\Gamma = \bigotimes_{\text{edges }e\in \Gamma} V_e
\end{equation*} 
where $V_e$ is a complex which depends only on the vertices adjacent to edge $e$.  Note $n\geq 1$ means at most one adjacent vertex is black and univalent.  A simple parity argument shows that if one of the adjacent vertices is black and univalent, then $V_e$ is acyclic and if none of the adjacent vertices is black and univalent then $V_e$ has cohomology of rank $1$ in degree $1$.  As such the cohomology of the associated graded complex is generated by those trees having no univalent and no bivalent black vertices, i.e. precisely by those generators of $\ZA$.  That is as graded vector spaces,
\begin{equation*}
H^\ast(Gr_\op{F}(Tw(\cB)(n))\cong \ZA(n)
\end{equation*}

Endowing $\ZA$ with the its degreewise filtration, the associated graded complex has $0$ differential.  Thus, the inclusion induces a quasi-isomprphism of associated graded complexes, and hence on the original spaces (see \cite{DR} Lemma A.3).
\end{proof}
\begin{corollary}  The cohomology of the reduced operad $\text{Tw}(\cB)$ is the gravity operad.
\end{corollary}

\section{Examples}\label{exsec}

The aim of this section is to further discuss examples of the above structures, interpreted in two different ways.  First, we give several direct applications of our constructions;  specifying a cyclic operad and a Maurer-Cartan element as input and drawing the obvious corollaries.  But we also specify several indirect applications; transferring the structure across (quasi)-isomorphisms.  The examples of indirect application raise the question of how to construct these operations directly in their respective environments.

The first two examples are members of a family of algebraic examples coming from the Koszul duality theory of (cyclic) operads via Definition $\ref{ccdef}$.

\begin{example}{\bf Hochschild and cyclic cohomology.}  The most fundamental example arising from Definition $\ref{ccdef}$ is the case where $\op{P}=\op{A}s$ and $A$ is a cyclic $A_\infty$ algebra.  Then
\begin{equation*}
\hat{CC}^\ast_{\op{A}s}(A)= \prod \Sigma(\fr{s}\op{A}s) \tensor^{S_n^+} Hom(A^{\tensor n}, A)\cong \prod_n \left(\Sigma^n sgn_{n+1}\tensor Hom(A^{\tensor n}, A)\right)^{\Z_{n+1}} 
\end{equation*} 
If $A$ is in particular Frobenius we have
\begin{equation*}
CC^\ast_{\op{A}s}(A)\cong \bigoplus_n \left(sgn_{n+1}\tensor Hom(A^{\tensor n}, A)\right)^{\Z_{n+1}} 
\end{equation*} 

Under the identification $Hom(A^{\tensor n}, A)\cong Hom(A^{\tensor n+1},\kk)$ we see precisely the set of `cyclic cochains' as defined by Connes \cite{Connes} often denoted $C^\ast_{\lambda}(A)$, and the Connes-Tsygan long exact sequence.  The BV structure on the Hochschild cohomology was established in \cite{Menichi} and \cite{Tradler}.  The homotopy BV structure was first given in \cite{Ward} in generality and by Kaufmann \cite{KCyclic} in the associative case.

\begin{corollary}  The cyclic cohomology of a Frobenius or cyclic $A_\infty$ algebra is a gravity algebra.  This structure is induced at the cochain level by the action of $\ZA$.
\end{corollary}
\end{example}

Our results tell us that the associative operad in the above example, can be replaced by any regular Koszul cyclic or anti-cyclic operad.  Here is another such example.

\begin{example}\label{diex} {\bf Dialgebra cohomology.}  Loday introduced the notion of a Leibniz algebras as a suitably non-commutative version of Lie algebras, and in turn introduced (associative) dialgebras, related via a commutator bracket \cite{lodi}.  The operad encoding dialgebras is anti-cyclic \cite{Cha} and we define the cyclic cohomology of a symplectic dialgebra using Definition $\ref{ccdef}$.  The cyclic cohomology of a  dialgebra is a gravity algebra, induced at the cochain level by the action of $\ZA$.  If the dialgebra is unital, we may apply Theorem $\ref{cycdelthm}$ to conclude that the operadic cochains form a $\op{BV}_\infty$ algebra.  This generalizes the results of \cite{MM}, which exhibits the homotopy Gerstenhaber structure.
\end{example}

\begin{example}  {\bf Deformation complexes.}  More generally, suppose we are given a morphism of cyclic operads $\phi\colon D(\op{N})\to \op{P}$, where $D$ is the bar construction \cite{GK}.  Such a $\phi$ specifies a MC element in the associated (cyclic) convolution Lie algebra, which is the Lie algebra associated to the convolution (cyclic) operad $\op{O}=Hom(\op{N}^\vee,\op{P})$ (where $\vee$ denotes the linear dual).  If the cyclic operads in question are regular and unital the above results apply to show that the associated deformation complex is an $\op{BV}_\infty$-algebra and the cyclic deformation complex is an $\ZA$-algebra.
\end{example}

\begin{example}\label{cyex} {\bf  Cyclic cohomology of Calabi-Yau algebras.}  The deformation theory of a Calabi-Yau algebra $A$ is controlled by a dg Lie algebra computing its negative cyclic homology \cite{CYVB}.  If $A$ is a Koszul Calabi-Yau algebra, then $A^!$, its Koszul dual, is Frobenius \cite{GinzCY}.  Given that
\begin{equation*}
HC^-_\ast(A)\cong HC^{-\ast}(A^!)
\end{equation*}
as Lie algebras (Theorem 35 of \cite{CY}), one can complete the Lie structure on the negative cyclic cohomology of $A$ to the structure of a gravity algebra, via our results above.  How to directly construct the higher brackets comprising this structure on the Calabi-Yau side is an open question.  The BV part of this story has been studied in greater detail, see \cite{GinzCY} and \cite{CY}.  Another interesting question is an operadic description of the up-to-homotopy structures on the cochain level of the Calabi-Yau side.
\end{example}

\begin{example}{\bf Hopf Algebroids.}  In \cite{Kow} Theorem 3.9, Kowalzig gives conditions on a left Hopf algebroid $U$ and a coefficient module $N$ such that the Hopf-cyclic cohomology $HC^\ast(U,N)$ is computed as the cyclic cohomology of a cyclic operad with multiplication, denoted $C^\ast_{\text{co}}(U,N)$.  We thus conclude, under the conditions specified in $loc.cit,$ that $C^\ast_{\text{co}}(U,N)$ is an $\ZA$ algebra inducing the structure of a gravity algebra on $HC^\ast(U,N)$.
\end{example}

\begin{example}\label{polycycex} {\bf Polydifferential cyclic complex.}  Following \cite{WillC} we let $M$ be a smooth oriented manifold with volume form $\omega$ and define $D_{\text{poly}}^n(M)$ to be the subcomplex of the Hochschild complex of $A:=C^\infty(M)$ consisting of those cochains which are differential operators in each argument, (after Kontsevich \cite{Kontdef}).  These spaces correspond to a suboperad of $End_A$ and we give this suboperad a cyclic operator by defining $t(\Psi)$ such that for compactly supported $f_i$ one has,
\begin{equation*}
\int_M f_0 (t\Psi)(f_1\cdc f_n) \ \omega = (-1)^n \int_M f_n \Psi(f_0\cdc f_{n-1}) \ \omega 
\end{equation*}
The complex of cyclic invariants, denoted $D_{\text{poly}}^{\ast+1}(M)^t$ is then an algebra over $\ZA$ inducing a gravity algebra structure on the polydifferential cyclic cohomology.  As mentioned in the introduction, this structure extends the Lie algebra that is the subject of Kontsevich's cyclic formality conjecture proven in \cite{WillC}.
\end{example}

\begin{example}\label{singex}  {\bf Singular cochains of an $S^1$-space.}  Given a topological space $X$, Gerstenhaber and Voronov showed \cite{GV} that its singular cochains $S^\ast(X)$ form an operad with multiplication, as we now recall.  We consider $\Delta_n=\{(s_0\cdc s_n) : \sum s_i=1, 0 \leq s_i\leq 1\}$.  For a choice of non-negative integers $n,m,i$ with  $1\leq i \leq n$ and an $n+m-1$ chain $\sigma\colon \Delta_{n+m-1}\to X$ we define
\begin{equation*}
 \sigma_{i,m}\colon \Delta_{m}  \to X \ \text{    by    } \ \sigma_{i,m}(s_0\cdc s_m) := \sigma(\underbrace{0\cdc 0}_{i-1},s_0\cdc s_m,\underbrace{0\cdc 0}_{n-i})
\end{equation*}
and
\begin{equation*}
 \sigma_{\hat{i},n}\colon \Delta_{n} \to X \text{    by    } \ \sigma_{\hat{i},n}(s_0\cdc s_n) := \sigma(s_0\cdc s_{i-1},\underbrace{0\cdc 0}_{m-1},s_{i}\cdc s_{n})
\end{equation*}

Define $\op{S}(n):= S^n(X)$.  For $\phi\in \op{S}(n), \psi\in \op{S}(m)$ we define:
\begin{equation}
\phi\circ_i\psi(\sigma) =  \phi(\sigma_{\hat{i},n})\cdot\psi(\sigma_{i,m})
\end{equation}
There is a morphism of operads $\op{A}s\to \op{S}$ induced by sending $\mu\in\op{A}s(2)$ to the $2$-cochain which is constantly $1$.
A compatible unit is given by the $1$-cochain which is constantly $1$.

We now assume $X$ has a continuous $S^1$-action, and give $\op{S}(X)$ the structure of a cyclic operad.  For an $n$-chain $\sigma$ we define

\begin{equation}
t(\sigma)(s_0\cdc s_n) = e^{2\pi i s_n}\cdot\sigma(s_n,s_0\cdc s_{n-1})
\end{equation}
and define $t^\ast\colon\op{S}(n)\to\op{S}(n)$ to be the linear dual of $t$.  Clearly $t^{n+1}=id$.

\begin{lemma} The pair $(\op{S},t^\ast)$ forms a unital cyclic operad with cyclic multiplication.
\end{lemma}
\begin{proof}  The proof is simple, so let us merely outline the process.  To verify that $(\op{S},t^\ast)$ is a cyclic operad one should work from the axiomatic definition (Definition $\ref{cycdef}$) although it is easier to verify the dual of the axioms.  To this end, for a fixed $n,m$ define $\circ^i(\sigma)=\sigma_{\hat{i},n}\tensor \sigma_{i,m}$.  Then for $i\leq n-1$, considering the dual of diagram $\ref{cycdefeq1}$ we may verify $\circ^i t_{n+m-1} = (t_n\tensor id)\circ^{i+1}$ by computing both sides.  The other case has $i=n$, in which we consider the dual of diagram $\ref{cycdefeq2}$ and verify that $\circ^n t_{n+m-1} = (t_m\tensor t_n)S_{\op{S}(n)\tensor\op{S}(m)}\circ^{1}$ by computing both sides.  This shows that $\op{S}$ is a cyclic operad and it remains to observe that $t\mu=\mu$ and thus the map $\op{A}s\to\op{S}$ defined above is a map of cyclic operads.  \end{proof}

The two complexes $\nsc{S}$ and $\cnsc{S}$ associated to the pair $(\op{S},\mu)$ compute $H^\ast(X)$ and $H^\ast_{S^1}(X)$ respectively (see \cite{Jones} Theorem 3.3, also \cite{Loday} p. 240).  On singular cohomology the product is the cup product, the brace operations encode Steenrod's $\cup_i$ products and the Gerstenhaber bracket is $0$ \cite{GV}.  The BV operator is that induced by the fundamental class of $S^1$ under the Kunneth theorem.  The associated long exact sequence is the Gysin sequence of the fibration $ES^1\times X \to ES^1\times_{S^1}X$.  On the chain level we see that the (normalized) singular cochains are an algebra over $CC_\ast(\op{C}acti^1)$.  In particular the BV operator is square-zero on the cochain level.  
\end{example}

\begin{example}\label{fcex}  {\bf The Fukaya category.}  Given that the Fukaya category $\op{F}(N)$ of a suitable symplectic manifold $N$ is a cyclic $\op{A}_\infty$ category \cite{FOO}, we can apply the above constructions to its endomorphism operad.  Here, in moving from cyclic $A_\infty$ algebras to categories, we observe that the operations defined above preserve the sequential matching of inputs.  The homotopy BV structure on $CH^{\ast}(\op{F}(N),\op{F}(N))$ was given in \cite{Ward} and requires unitality.  We now record the following corollary of our above work.

\begin{corollary}  The Lie bracket on the cyclic cochains of the Fukaya category is the first of a family of higher brackets which assemble to a gravity algebra on the cyclic cohomology.  This structure is induced on the cochain level by the action of $\ZA$. 
\end{corollary}

To relate our results to the Getzler's original construction of a gravity algebra on the equivariant cohomology one should appeal to Costello's theorem \cite{Co} relating TCFTs and cyclic $A_\infty$ categories.
\end{example}

\begin{example}\label{kiex} {\bf de Rham chains and string topology.}  The string topology constructions of Chas and Sullivan \cite{CS} were one inspiration for this work.  To relate our results to string topology requires a suitable algebraic model and we discuss that of Irie given in \cite{Irie1}.  In $loc.cit$, the author introduces the `de Rham chain complex' of a closed oriented Riemannian manifold, denoted $C_\ast^{LP}(M)$ as a hybrid of the singular chains and usual de Rham complex which is suitable to define chain level versions of string topology operations.  The homology of $C_\ast^{LP}(M)$ is the homology of the free loop space and one may define a chain level multiplication, BV operator and Lie bracket on $C_\ast^{LP}(M)$ inducing the Chas-Sullivan BV structure on homology.

The complex $C_\ast^{LP}(M)$ is of the form $(\nsl{C}, \delta_\eta)$ for a certain cyclic operad $\op{C}$ with associative multiplication $\eta$.  This fact is implicit in the author's subsequent paper \cite{Irie2} and was pointed out to me by him.  We may therefore see (the normalized subcomplex of) $C_\ast^{LP}(M)$ as an algebra over $\op{TS}_\infty$ (and thus $\op{BV}_\infty$).  It should be possible to show further that the homotopy gravity structure on $(\cnsl{C}, \delta_\eta)$ given by Theorem C induces the original homology level construction of \cite{CS}.
\end{example}

\begin{example}\label{cmex} {\bf String topology of classifying spaces.}

In parallel with the string topology of manifolds, Chataur and Menichi have described string topology operations on the homology of the free loop space of $BG$, for a finite discrete group or compact Lie group $G$ \cite{ChM}, see also \cite{Ch}.  These operations are also modeled conjecturally (Conjecture 39 of \cite{ChM}) by Hochschild and cyclic cohomology.  Namely, via the Burghelea-Fiedorowicz, Goodwillie isomorphisms $H^{\ast+d}(LBG;\kk)\cong HH^\ast(S_\ast(G), S_\ast(G))$ and $H_{S^1}^\ast(LBG;\kk)\cong HC^\ast(S_\ast(G))$ \cite{BF}, \cite{Good}.  When $S_\ast(G)$ is quasi-isomorphic to a Frobenius algebra ({\it eg} if $G$ is finite) the BV and gravity algebras arise on the algebraic side via a cochain level action of $\TA$ and $\ZA$.

\end{example}

\appendix

\section{Trees}\label{sec:trees}  Axiomatic treatment of graph theory is available elsewhere ({\it eg} \cite{KW}) and we assume terminology such as graph, vertex, edge, valence, etc. is understood.  In general one often considers graphs with half-edges called flags.  Flags not part of an edge are called tails.  Here is a graph with $3$ vertices, $2$ edges, $7$ tails and $11$ flags.
\begin{equation}\label{g1}
	\includegraphics[scale=.15]{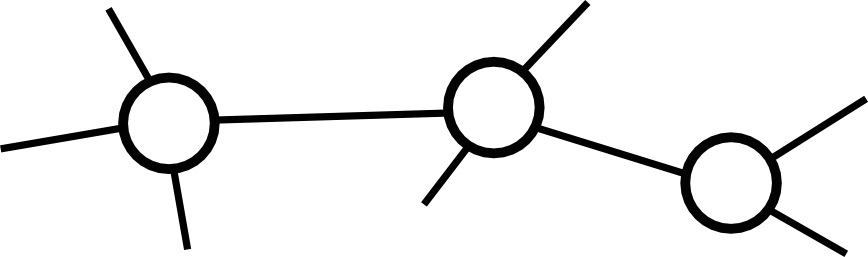}
\end{equation}

A tree is a contractible graph.  For the purposes of this paper, when using the terminology `tree' we mean a tree with no non-root tails, and we will use the terminology `tree with tails' if tails are needed.  Trees may be non-planar (left) or planar (center).  Our planar trees come with a cyclic order on the flags adjacent to each vertex.  A `labeled tree' refers to a labeling of the vertices by $1\cdc n$ (right).
\begin{equation}\label{g2}
	\includegraphics[scale=.3]{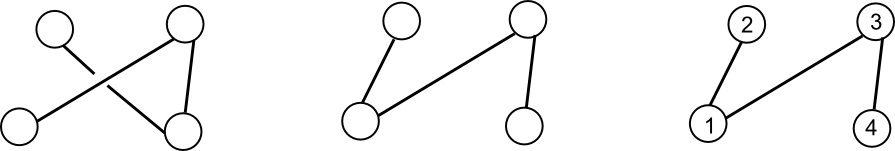}
\end{equation}

A directed graph has an orientation of each edge (left).  A non-planar tree with a globally consistent direction is called non-planar rooted.
We depict the direction with a $\downarrow$ (right).  Note this arrow is not an edge.  It can be considered a tail and counts toward the valence of a vertex.
\begin{equation}\label{g3}
	\includegraphics[scale=.3]{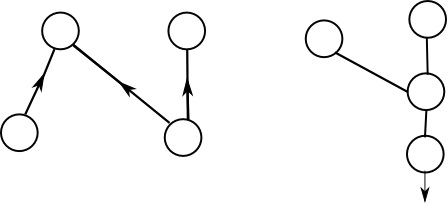}
\end{equation}
In a rooted tree there is a notion of height of a vertex by counting the number of vertices on the unique path to the root.  We take the convention that the rooted vertex has height $1$.  A vertex of maximum height is called a leaf.  More generally a vertex of valence $1$ is called a leaf.

In a tree we can contract an edge and identify its adjacent vertices to form another tree.  More generally we may contract a subtree.  In the labeled context we can contract a subtree having sequential labels.  Our notation for contracting a subtree $T^\prime$ in $T$ is $T/T^\prime$.  Here is a contraction of the subtree spanned by $\{2,3,4\}$:

\begin{equation}\label{g10}
	\includegraphics[scale=.3]{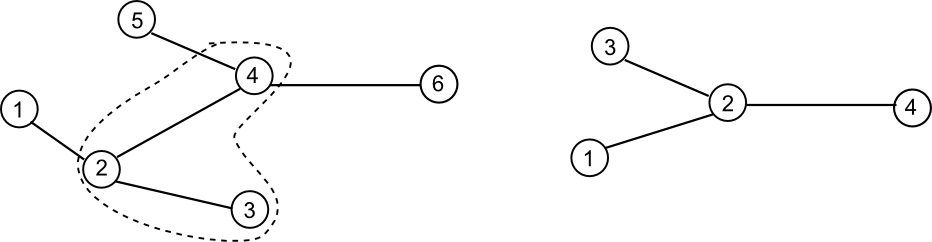}
\end{equation}

A two-color graph comes with a subdivision of the set of vertices into the `black' and `white' vertices.  We call an edge black (resp. white) if both adjacent vertices are black (resp. white).  We call a vertex unstable if it has valence $\leq 2$.  A black and white tree (abbr. b/w tree) is a two-color tree with labeled white vertices having no unstable black vertices and no black edges.  In the rooted case, the root counts toward the valence and informs stability.  Left and left-center are not b/w trees. Right-center is a planar b/w tree.  Right is a planar rooted b/w tree.

\begin{equation}\label{g4}
	\includegraphics[scale=.3]{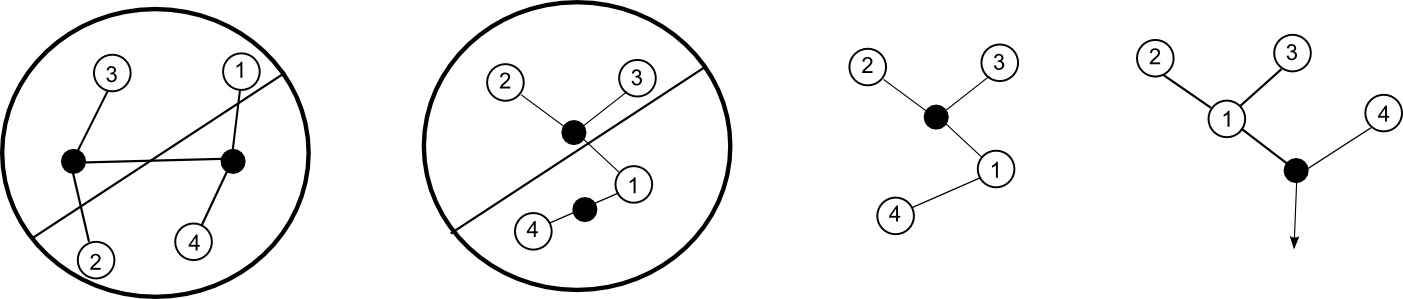}
\end{equation}

In a planar tree, an angle of a vertex is the arc joining two cyclically adjacent flags.  In a b/w tree we use the terminology white (resp. black) angle.  Left is a tree with 5 white angles and 3 black angles.  In a planar tree, the choice of an angle determines a root.  Center is a choice of black root.  Right is a choice of white root.  Due to stability issues, not every rooted b/w tree arises from the choice of an angle of a b/w tree.

\begin{equation}\label{g5}
	\includegraphics[scale=.3]{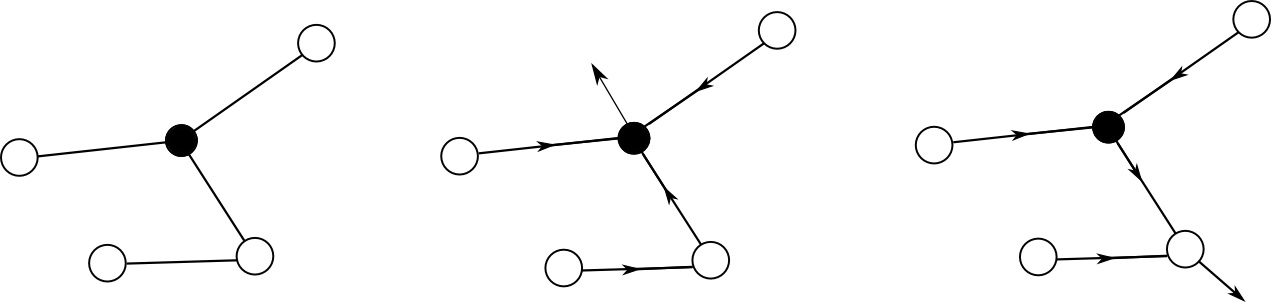}
\end{equation}

The branches of a vertex are the objects left after deleting the vertex.  In particular the branches of a vertex are in bijective correspondence with its set of adjacent flags.  In a b/w tree we may contract a white angle at a non-leaf vertex to form another b/w tree; before is left, after is right.  The process can be described as removing the corresponding branches, gluing them together with a black vertex, and then reattaching.  This operations is associative with respect to consecutive angles.

\begin{equation}\label{g55}
	\includegraphics[scale=.3]{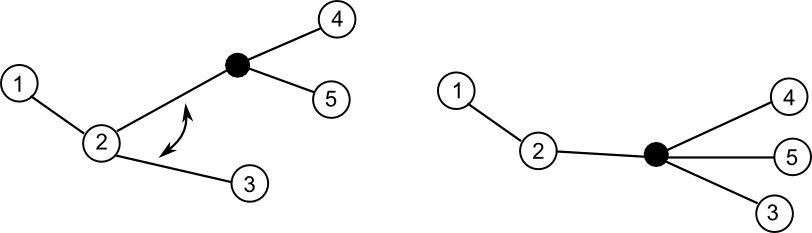}
\end{equation}

An $A_\infty$ labeled tree is a b/w tree along with the additional data of a label of each black vertex by a planar tree whose tails correspond to the flags at said vertex.  Left is an $A_\infty$ labeled tree with one black vertex.  A rooted $A_\infty$ labeled tree with black root is center and with white root is right. 

\begin{equation}\label{g6}
	\includegraphics[scale=.3]{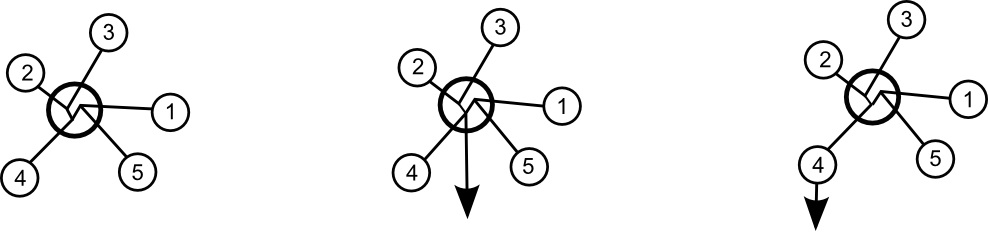}
\end{equation}

The terminology `$A_\infty$ labeled tree' arises from the fact that a rooted $A_\infty$ labeled tree has black vertices labeled by cells of the $\op{A}_\infty$-operad (below).  As such it makes sense to talk about the degree of a black vertex label in such a tree.

\begin{equation}\label{g8}
	\includegraphics[scale=.3]{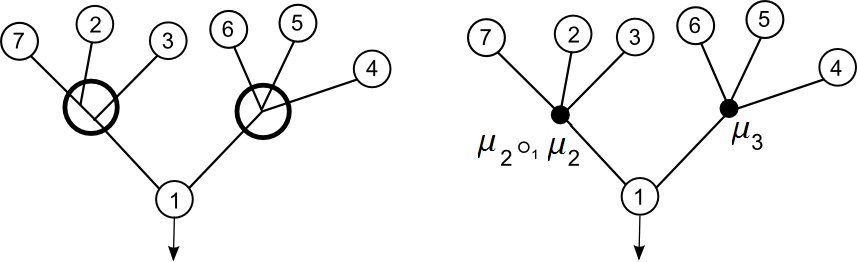}
\end{equation}
Occasionally, we will consider a construction with b/w trees that creates black edges or black vertices of arity $2$.  If so we will contract the black edges and we will erase those unstable black vertices.  Notice that contracting black edges also makes sense for $A_\infty$ labeled trees:
\begin{equation}\label{g7}
	\includegraphics[scale=.3]{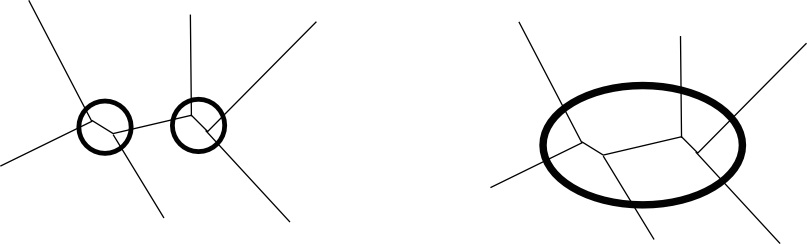}
\end{equation}

Contracting white angles also makes sense for an $A_\infty$ labeled tree by multiplying the labels.  In the $A_\infty$ labeled context contracting adjacent angles is not an associative operation.  We may contract multiple angles simultaneously using higher corollas to interpolate.  Here is a simultaneous contraction of two angles.

\begin{equation}\label{g9}
	\includegraphics[scale=.3]{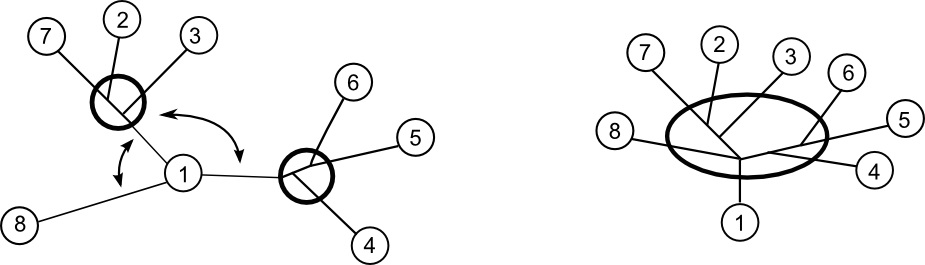}
\end{equation}

\section{Recollection of cyclic operads} \label{sec:cyc}

The purpose of this section is to briefly review cyclic operads, along with anti-cyclic and odd cyclic variants, and to give several fundamental examples.

\subsection{Cyclic operads: definition}
Define $S_n^+$ to be the group of permutations of $\{0,\dots,n\}$.  The symmetric group $S_n$ is viewed as the subset of $S_n^+$ fixing $0$.  Define $t_n$ to be the permutation $(0 \dots n)\in S_n^+$ and let $\Z_n^+\subset S_n^+$ be the subgroup generated by $t_n$.

\begin{definition}\label{cycdef} \cite{GeK1}  A cyclic operad $\op{C}$ is an operad along with an action of $S_n^+$ on each $\op{C}(n)$, agreeing with the underlying operad structure on the subgroup $S_n$, such that the following three axioms are satisfied: 
\begin{enumerate}
\item If $\nu \colon \kk \to \op{C}(1)$ denotes the operad unit then $t_1 \circ \nu = +\nu$.
\item The following diagram commutes for all $1\leq i \leq m-1$.

\begin{equation}\label{cycdefeq1}
\xymatrix{\op{C}(m)\tensor \op{C}(n) \ar[rr]^{\circ_i} \ar[d]_{t_m\tensor id} && \op{C}(n+m-1) \ar[d]^{t_{n+m-1}} \\ \op{C}(m)\tensor\op{C}(n) \ar[rr]^{\circ_{i+1}} && \op{C}(n+m-1)}
\end{equation}
\item The following diagram commutes.

\begin{equation}\label{cycdefeq2}
\xymatrix{\op{C}(m)\tensor \op{C}(n) \ar[rr]^{\circ_m} \ar[d]_{S\circ (t_m\tensor t_n)}   && \op{C}(n+m-1) \ar[d]^{t_{n+m-1}}  \\ \op{C}(n)\tensor\op{C}(m) \ar[rr]^{\circ_1} && \op{C}(n+m-1)}
\end{equation}
Here $S$ denotes the commutator $\op{C}(m) \tensor \op{C}(n) \stackrel{\cong}\longrightarrow \op{C}(n) \tensor  \op{C}(m)$.  
\end{enumerate}
\end{definition}

\begin{definition} \label{nonsigmacycdef}  A non-$\Sigma$ cyclic operad is a non-$\Sigma$ operad $\op{C}$ along with an action of $\Z_n^+$ such that axioms 1-3 of Definition $\ref{cycdef}$ are satisfied.
\end{definition}

These definitions make sense in any symmetric monoidal category, but we restrict our primary attention to dg vector spaces.  The endomorphism operad is not cyclic in general.  However, we say a dg vector space is cyclic if it comes with a symmetric nondegenerate inner product and if $A$ is a cyclic dg vector space then $End_A$ is a cyclic operad, using the inner product to identify $A$ with its linear dual.  When saying $A$ is an algebra over a cyclic operad $\op{C}$ one presupposes both that $A$ is cyclic and that there is an operad morphism $\op{C} \to End_A$ which is $S_n^+$ equivariant.

\subsection{Examples of cyclic operads} 

\begin{example}\label{Frobeniusalgebra} (Frobenius Algebras)   The associative operad $\op{A}s$ (resp. the commutative operad $\op{C}om$) is cyclic by defining the action of $(0\dots n)$ in arity $n$ to be the identity. A cyclic algebra over $\op{A}s$ (resp. $\op{C}om$) is a symmetric non-commutative (resp. commutative) Frobenius algebra.
\end{example}

\begin{example}\label{liecyclic}  The fact that $\op{L}ie$ has a coherent $S_{n+1}$ action was first observed by Kontsevich \cite{Kont2}, and the fact that $\op{L}ie$ is cyclic follows from the fact that the quadratic dual of $\op{L}ie$ is cyclic \cite{GeK1}.  In particular Getzler and Kapranov show that there is an isomorphism of $S_{n+1}$-modules $\op{L}ie(n)\tensor V_{n,1}\cong \op{L}ie(n+1)$, where $V_{n-1}$ is the hyperplane representation.
\end{example}

\begin{example}\label{cycaex} (cyclic $A_\infty$ algebras)  The cells of the $\op{A}_\infty$ operad are indexed by planar rooted trees with tails.  Labeling the root by $0$ and the other tails in the planar order gives an action of $\Z_n^+$ which makes $\op{A}_\infty$ a cyclic operad.  A cyclic $A_\infty$ algebra is an algebra over the cyclic operad $\op{A}_\infty$.  The $\Z_n^+$ action on $\mu_n$ is given by $\mu_n\mapsto (-1)^{n}\mu_n$, but the rotation action for operadic compositions is non-trivial in general.  A utility of cyclic operads in describing cyclic $A_\infty$ algebras is to efficiently encode the infinitely many axioms of invariance.
\end{example}

\begin{example}\label{cycliclex} (cyclic $\op{L}_\infty$ algebras)  Let $\op{L}_\infty:=D(\op{C}om)$ as in \cite{GK}.  The contravariant functor $D$ takes cyclic operads to cyclic operads \cite{GeK1}.  Applying $D$ to the standard map $\op{A}s\to \op{C}om$ induces a map of cyclic operads $\op{L}_\infty\to\op{A}_\infty$.  This tells us that the $\mathbb{Z}_{n}^+$-action on a generator $\ell_n$ is by multiplication by the sign $(-1)^{n}$.
\end{example}

\subsection{Anti/odd cyclic operads and suspension.}  In this subsection we will recall anti-cyclic and odd cyclic operads.  Anti-cyclic operads originated in \cite{GeK1}.  Odd structures were emphasized in \cite{KWZ}.

\begin{definition}\label{anticycdef}   Let $\op{C}$ be an operad.  We say $\op{C}$ is anti-cyclic if Definition $\ref{cycdef}$ holds after making the following two alterations; replace the $+$ sign with a $-$ sign in the unit axiom and ask diagram $\ref{cycdefeq2}$ to anti-commute.
\end{definition}

Recall the determinant operads $\Lambda^{\pm}:=\fr{s}^{\pm}\op{C}om$.  Extending the $S_n$ action to an $S_n^+$ action by the sign representation yields the following lemma.

\begin{lemma}\label{detcyc}  The determinant operads are naturally anti-cyclic operads.  The tensor product of a cyclic operad and an anti-cyclic operad is anti-cyclic.  The tensor product of two cyclic or two anti-cyclic operads is cyclic.  In particular, the operadic suspension or desuspension of a cyclic operad is naturally an anti-cyclic operad and vice versa.
\end{lemma}

\begin{example} \label{anticycexample}  The suspension of a cyclic vector space is a symplectic vector space.  If then follows from Lemmas $\ref{MSSlemma}$ and $\ref{detcyc}$ that the endomorphism operad of a symplectic vector space $V$ is anti-cyclic.  We say $V$ is a (symplectic) algebra over the anti-cyclic operad $\op{P}$ if there is a map of anti-cyclic operads $\op{P}\to End_V$.
\end{example}

\begin{example}  The operad encoding pre-Lie algebras is anti-cyclic \cite{Cha}.  The fact that operadic suspension takes anti-cyclic operads to cyclic operads (see Lemma $\ref{detcyc}$) thus implies the operad of symmetric brace operations is cyclic.
\end{example}

In analogy with the operad case we now define odd cyclic operads.

\begin{definition}\cite{KWZ}  Let $\op{C}$ be an $\mathbb{S}^+$-module such that $\Sigma^{-1}\fr{s}^{-1}\op{C}$ is a cyclic operad.  Then we say $\op{C}$ is an odd cyclic operad.
\end{definition}

As in the operad case we view $\Sigma\fr{s}$ as the `oddification' functor for cyclic operads.  The fundamental Lie algebras which we will consider come from first moving to an odd structure in this way.  The importance of the odd gluings is again masked by the fact that every odd cyclic operad arises from a cyclic operad in this way.

\subsection{Graph and triple interpretation}

$\it{A \ priori}$ the structure maps in a cyclic operad are of the form $\circ_i$; thought of in the tree picture as gluing the $0$ leg of the second tree to the $i^{th}$ leg of the first tree, with $i \neq 0$.  We can, however, use the action of $S_n^+$ to define structure maps gluing any two legs of two given trees.  As such cyclic operads can be thought of as having operations encoded by all trees, and this intuition can be made precise via a triple interpretation: cyclic operads can be defined as algebras over a triple of (non-rooted) trees with labeled leaves \cite{GeK1}.

Similarly, anti and odd cyclic operads should be thought of as having operations corresponding to decorated trees.  For anti-cyclic operads the decorations are an orientation of each edge and flipping an orientation produces a $-$ sign.  For odd-cyclic operads the decoration is an order on the set of edges, and permuting the order by an odd permutation produces a $-$ sign.  Once again we can give the alternate definition of these structures as algebras over the corresponding triple.

One subtlety in establishing the equivalence of the two definitions concerns relabeling.  The rooted structure when dealing with operads gives a canonical way to relabel the leaves of a rooted tree after grafting.  However, there is no canonical way to relabel in a non-rooted context for an arbitrary grafting.  As such we must make a choice when extending the defining operations to arbitrary edge graftings (see Definition $\ref{circmaps}$).  In the end, however, the structures that we care about will be independent of these choices.

\begin{definition} \label{circmaps}  Let $\op{C}$ be a cyclic operad, anti-cyclic operad, or odd cyclic operad and let $a \in \op{C}(n)$ and $b \in \op{C}(m)$.  We first define $a\circ_0 b=t_n(a)\circ_1 b$. We then define the `edge-grafting' operation $\circ_{ij}$ by:
\begin{eqnarray*}
\circ_{ij} \colon \op{C}(n) \tensor \op{C}(m) &\to & \op{C}(n+m-1) \\
a \tensor b &\mapsto& a\circ_i t_m^{-j} b
\end{eqnarray*}
\end{definition}

The choice of extension is natural up to cyclic permutation (see Remark $\ref{circijrmk}$).  The extension given here is chosen so that $a\circ_{i0}b=a\circ_ib$.

\begin{lemma}\label{extlem}  The $\circ_{ij}$ operations are commutative up to cyclic permutation.  Precisely, in the notation of Definition $\ref{circmaps}$, 
\begin{equation*}
a\circ_{ij}b = (-1)^{\epsilon} t_{n+m-1}^{m+i-j}(b\circ_{ji}a)
\end{equation*}
where
\begin{equation}
\epsilon = \begin{cases} |a||b| & \text{if } \op{C}\text{ is cyclic} \\ |a||b|+1 & \text{if $\op{C}$ is anti-cyclic} \\ (|a|-1)(|b|-1)+1 & \text{if $\op{C}$ is odd cyclic}   \end{cases}
\end{equation}
\end{lemma}

\begin{proof}  A direct calculation from the axioms (Definition $\ref{cycdef}$).
\end{proof}

\begin{remark}\label{circijrmk}
The definition of the operation $\circ_{ij}$ depended on a non-canonical choice but all choices differ only by a cyclic permutation, hence the induced operation on (co)invariants is natural.  Moreover there is a natural operation corresponding to every planar tree with leaves having labeled flags at each vertex (plus additional decorations in the odd and anti case as discussed above).  Such trees having one edge and two vertices correspond to $\circ_{ij}$.  We define the action on a tree with multiple edges is the composition of the $\circ_{ij}$ operations, one for each edge.  
\end{remark}

\bibliography{bib}
\bibliographystyle{alpha}

\end{document}